\documentclass[twocolumn]{autart}    

\usepackage{graphicx}
\usepackage{subcaption}
\usepackage{amsmath}
\usepackage{mathtools}
\usepackage{amssymb}
\usepackage{amsfonts}
\usepackage{mathrsfs}
\usepackage{dsfont}
\usepackage{multirow}
\usepackage{tikz}
\usepackage{etoolbox}
\usepackage{epstopdf}
\usepackage{algorithmic}
\usepackage{amsopn}
\usepackage{lipsum}
\usepackage{blindtext}
\usepackage{enumitem}

\newtheorem{theorem}{Theorem}
\newtheorem{lemma}{Lemma}
\newtheorem{definition}{Definition}
\newtheorem{assumption}{Assumption}
\newtheorem{remark}{Remark}

\newtheorem{problem}{Problem}
\makeatletter
\@ifundefined{proof}{%
  \newenvironment{proof}[1][\proofname]{%
    \par\noindent\textbf{#1.}\ \ignorespaces
  }{%
    \hfill$\square$\par\medskip
  }%
}{}
\providecommand{\proofname}{Proof}
\makeatother

\DeclareMathOperator{\diag}{diag}
\DeclareMathOperator*{\argmin}{arg\,min}
\DeclareMathOperator*{\subjectto}{s.t.}
\DeclareMathOperator*{\col}{col}

\begin{document}

\begin{frontmatter}

\title{Safe distributed generalized Nash equilibrium seeking via control barrier functions\thanksref{footnoteinfo}}

\thanks[footnoteinfo]{This work was funded by the Natural Sciences and Engineering Research Council (NSERC) of Canada [RGPIN-2024-04958]. The material of this paper was partially presented at the 23rd IFAC World Congress, August 23-28, 2026, Busan, Republic of Korea. Corresponding author L. Pavel.}

\author[UofT]{Yihan Meng}\ead{yihan.meng@mail.utoronto.ca},
\author[ND]{Weijian Li}\ead{wli26@nd.edu},
\author[UofT]{Lacra Pavel}\ead{pavel@ece.utoronto.ca}

\address[UofT]{Department of Electrical and Computer Engineering, University of Toronto, Toronto, ON M5S 3G4, Canada}
\address[ND]{Department of Electrical Engineering, University of Notre Dame, Notre Dame, IN 46556, USA}


\begin{abstract}
In this paper, we consider generalized Nash equilibrium (GNE) seeking in non-cooperative games with coupled constraint sets. Specifically, we aim to enforce safety for distributed GNE seeking, whereby the safety specifications are encoded in the coupled constraint set. To achieve this, we introduce the control barrier function (CBF) in the design of the GNE seeking dynamics. We design the dynamics for both full- and partial-information settings, where each player has knowledge of the decision information of all other players or only neighboring players, respectively. We justify the proposed dynamics by showing that the coupled constraint set is forward invariant, the equilibrium of the dynamics coincides with the exact GNE of the game, and the dynamics is asymptotically stable. Furthermore, we extend the approach to games where the agents are multi-integrators. Numerical simulations are provided to verify our results.
\end{abstract}

\end{frontmatter}

\section{Introduction}\label{sec:introduction}
The generalized Nash equilibrium (GNE) problem, first introduced by Debreu \cite{Debreu}, is an extension of the celebrated Nash equilibrium (NE) problem \cite{nash}, where the feasible set of each player is affected by the other players. A GNE is a collective decision profile, such that no one in the game can further improve its objective by unilaterally deviating from the current decision.
The GNE seeking problem has garnered increasing interest in recent years \cite{bianchi2021,pavel2019,romano2022}. 
While reaching a GNE ensures that coupled constraints are satisfied at steady-state, it does not guarantee their satisfaction during the transient state. In many applications, violations of the coupled constraints may jeopardize system safety \cite{hall2025limits,ma2011decentralized,wang2014generalized,wang2025online}. The GNE seeking problem is further complicated when the trajectory of decisions produced by the GNE seeking dynamics needs to be maintained within the coupled constraint set where the safety specifications are encoded. 

The most adopted GNE seeking approach is the primal-dual method and its variants, including discrete-time schemes \cite{peng2019}, fully-distributed algorithms \cite{bianchi2021} and passivity-based generalizations \cite{weijian2024}. However, the traditional primal-dual dynamics makes no safety guarantees. In order to enforce safety in GNE seeking, one resorts to interior-point methods. By incorporating the constraints into players' objectives with a log-barrier function, an inexact-penalty approach was proposed in \cite{romano2020}, and was later improved by \cite{romano2022} to achieve exact convergence to the GNE. However, the penalty-based method can be inefficient: as the log-barrier shifts the tracked instantaneous equilibrium away from the exact one, the trajectory may transiently diverge. 

On the other hand, control barrier functions (CBFs) have emerged as a prominent approach for safety enforcement \cite{allibhoy2023,allibhoy2025,ames2016,ren2023vector}. A CBF-based safe gradient flow for constrained nonlinear programming problems was proposed in \cite{allibhoy2023}, and was later extended to solve for variational inequalities \cite{allibhoy2025} and distributed optimization with coupled constraints \cite{mestres2023distributed}.
However, since the GNE problem is inherently coupled through both players' objectives and constraint sets, integrating CBFs into distributed GNE seeking dynamics remains a challenge. 

On the other hand, most existing GNE seeking algorithms assume full-information setting where all players' decisions are available to everyone \cite{paccagnan2016distributed,peng2019,zhu2016distributed}. For distributed systems, a partial-information structure, where players can only access their neighbors' decisions through local communication, is more realistic \cite{gadjov2022exact,pavel2019,romano2020}. 
Additionally, players may have intrinsic dynamics which need to be considered. Existing approaches leverage primal-dual methods \cite{liu2023generalized} or penalty-based algorithms \cite{romano2020,romano2022}, but these methods are in short of either safety considerations or computational efficiency.

Motivated by the above observations, this paper presents a CBF-based GNE seeking dynamics that guarantees safety throughout the entire evolution. Our contributions are threefold:
\begin{enumerate}
    \item We design a safe GNE-seeking dynamics using a CBF-based control synthesis law that restricts the dual variables to an admissible set, thereby guaranteeing safety and preserving asymptotic convergence to the exact GNE. Unlike the single-agent algorithms \cite{allibhoy2023,allibhoy2025} and distributed optimization in \cite{mestres2023distributed}, our design is applicable to non-cooperative games with coupled constraint sets.
    \item We extend the CBF-based safe GNE seeking dynamics to partial-information settings by incorporating consensus-based estimation to compensate for the lack of global decision information. To the best of our knowledge, this is the first work to address this problem in the literature.
    \item We also consider dynamical players, specifically multi-integrators, and introduce a safe GNE seeking control input design via a CBF-based coordinate transformation that translates the constraint invariance of the decision variables to that of their higher-order derivatives. 
\end{enumerate}

A preliminary version of this work appeared in \cite{Meng2026}. In contrast, this current paper extends the safe GNE seeking dynamics to partial-information settings and presents a controller design method for multi-integrator agents, generalizing the results in \cite{Meng2026}.

The remainder of this paper is organized as follows. Section \ref{sec:preliminaries} introduces relevant preliminaries and formulates the GNE seeking problem. Section \ref{sec:single-int} proposes a CBF-based distributed algorithm and provides its convergence analysis. Section \ref{sec:partialinfo} adapts the design to a partial-information setting. Section \ref{sec:multi-int} extends the algorithm to GNE seeking with multi-integrators. Section \ref{sec:sim} verifies the proposed approach with simulation examples. Section \ref{sec:conclusion} concludes this paper and outlines future research directions.

\textit{Notations}: The symbols $\mathbb{R}^n$ and $ \mathbb{R}^n_+$ denote the $n$-dimensional real vectors and nonnegative real vectors, respectively. $\mathds{1}_n$ and $0_n$ are the $n$-dimensional vectors with all entries being $1$ and $0$, respectively. $\|\cdot\|$ is the Euclidean norm and $\otimes$ is the Kronecker product. col$\{x_i\}_{i\in\mathcal{I}}$ means to stack all vectors $x_i$, $i\in\mathcal{I}= \{1,2,\cdots,N\}$, into one column vector, and diag$\{X_i\}_{i\in\mathcal{I}}$ is a diagonal matrix with diagonal entries $X_i$, $i\in\mathcal{I}$, where $X_i$ can be numbers or matrices, moreover, diag$\{v\}$ where $v\in\mathbb{R}^n$ gives an $n\times n$ diagonal matrix with diagonal entries being the elements of vector $v$. For a vector $v\in\mathbb{R}^n$, $v\geq0$ means $v$ is nonnegative entrywise, and $v\perp w$ means the vectors $v$ and $w$ are perpendicular to each other, i.e., $v^\top w=0$. For a differentiable function $f:\mathbb{R}^m\to\mathbb{R}^n$, $\frac{\partial f(x)}{\partial x}$ denotes its Jacobian matrix. For $J:\mathbb{R}^n\times\mathbb{R}^m\to\mathbb{R}^p$, $\nabla_x J(x,y)$ denotes the partial gradient of $J$ with respect to $x$. Let $F:\mathbb{R}^n\to\mathbb{R}^n$ and $\mathcal{D}\subset\mathbb{R}^n$, then a variational inequality is a problem of finding $x^*\in\mathcal{D}$ such that $(x-x^*)^\top F(x^*)\geq 0, \forall x\in\mathcal{D}$, and it is referred to with the notation VI$(\mathcal{D},F)$. For a matrix $A$, $\sigma_{\max}(A)$ is the largest eigenvalue of $A$. We denote $[A]_{ij}$ as the entry of matrix $A$ at $i$th row and $j$th column, and $[v]_i$ as the $i$th element of vector $v$. For $x\in\mathbb{R}^n$, $x^{(r)}$ denotes its $r$th order time-derivative of $x$.
\section{Preliminaries and game setup}\label{sec:preliminaries}

\subsection{Preliminary concepts}
\textit{Graph theory.} The communication among the players is modeled by an undirected graph $\mathcal{G}(\mathcal{I},\mathcal{E},\mathcal{A})$, where $\mathcal{I} = \{1,2,\cdots,N\}$ is the index set of all the nodes, $\mathcal{E}$ is the edge set such that $\mathcal{E} \subseteq \mathcal{I}\times\mathcal{I}$, and $\mathcal{A} = [a_{ij}]\in\mathbb{R}^{N\times N}$ is the adjacency matrix of the graph such that $a_{ij} = a_{ji}$, $a_{ij}>0$ if and only if $(i,j)\in\mathcal{E}$, and $a_{ij}=0$ otherwise. The graph $\mathcal{G}(\mathcal{I},\mathcal{E},\mathcal{A})$ contains no self-loop, i.e., $a_{ii}=0$ for all $i$. The neighbor set of player $i$ is defined as $\mathcal{N}_i = \{j|(i,j)\in\mathcal{E}\}$. The Laplacian matrix is defined as $\mathcal{L} = \text{diag}\{\mathcal{A}\cdot\mathds{1}_N\}-\mathcal{A}$, and by its definition, we have that $\mathcal{L}\cdot\mathds{1}_N=0_N$ and $\mathcal{L} = \mathcal{L}^\top$. Moreover, if the graph is connected, then rank$(\mathcal{L})=N-1$, meaning $\mathds{1}_N$ is the eigenvector of $\mathcal{L}$ corresponding to the simple eigenvalue $0$.

\textit{Control barrier functions.} Consider an affine control system
\begin{equation}\label{cbf-affineSys}
  \dot{x} = f(x)+h(x)u,
\end{equation}
where $f:X\to\mathbb{R}^n$ and $h:X\to\mathbb{R}^{n\times m}$ are locally Lipschitz, $x\in X\subset\mathbb{R}^n$ and $u\in\mathcal{U}\subset\mathbb{R}^m$. A vector control barrier function is defined as follows:
\begin{definition}\cite{allibhoy2023,ren2023vector}\label{definition-vcbf}
  Given a constraint set \( \mathcal{D} \subset X\subset \mathbb{R}^n\). A vector control barrier function (VCBF) is a continuously differentiable function $g:\mathbb{R}^n\to\mathbb{R}^m$ such that:
  \begin{enumerate}
    \item The constraint set can be expressed by 
    \[\mathcal{D} = \{x\in\mathbb{R}^n\mid g(x)\leq 0_m\};\]
    \item There exists a constant $\gamma >0$ such that the admissible set
    \begin{multline}\label{zcbf-soln}
    K_\gamma(x) = \left\{ u \in \mathcal{U} \; \left| \; \frac{\partial g(x)}{\partial x} f(x)\right.\right. \\
    \left.+ \frac{\partial g(x)}{\partial x} h(x) u + \gamma g(x) \leq 0_m \right\}
\end{multline}
    is non-empty.
  \end{enumerate}
\end{definition}
Then we have the following lemma on the forward invariance of the constraint set $\mathcal{D}$:
\begin{lemma}\cite{allibhoy2023,ren2023vector}\label{lemma-zcbf}
  Consider the system (\ref{cbf-affineSys}) with constraint set $\mathcal{D}$, and let $g$ be a VCBF. Then any feedback controller $u:X\to\mathcal{U}$ satisfying $u(x)\in K_\gamma(x)$ for all $x\in X$ and $(f(x)+h(x)u(x))$ being locally Lipschitz renders $\mathcal{D}$ forward invariant, i.e., if $x(0)\in\mathcal{D}$, then $x(t)\in\mathcal{D}$ for all $t\geq0$.
\end{lemma}

\subsection{Game setup}

We consider a generalized game problem with $N$ players indexed by $\mathcal{I} = \{1,2,\cdots,N\}$. Each player $i\in\mathcal{I}$ chooses its decision $x_i\in\mathbb{R}^{n_i}$ to minimize its local objective function $J_i(x_i,x_{-i}):\mathbb{R}^{n_i}\times\mathbb{R}^{n-n_i}\to\mathbb{R}$, dependent both on its own decision $x_i$ and the decisions $x_{-i}$ of the other players, where $x_{-i} = \text{col}\{x_k\}_{k\in\mathcal{I}\backslash\{i\}}$ is the stacked vector of the other $N-1$ players' decision variables and $\sum_{i\in\mathcal{I}}n_i=n$. The choice of player $i$'s decision is confined to a coupled feasible set $X_i(x_{-i})= \{x_i\in\mathbb{R}^{n_i}\mid(x_i,x_{-i})\in\mathcal{X}\}$, where $\mathcal{X} = \{x\in\mathbb{R}^n\mid Gx-H = \sum_{i\in\mathcal{I}}(g_ix_i-h_i)\leq 0_m\}$ is separable and $g_i\in\mathbb{R}^{m\times n_i}$, $h_i\in\mathbb{R}^m$ are known only to player $i$. The resulting game can be formulated as
\begin{equation}\label{problem-formulation}
    \min_{x_i\in \mathbb{R}^{n_i}}\; J_i(x_i,x_{-i}),\quad\subjectto\; (x_i,x_{-i})\in \mathcal{X}.
\end{equation}
We aim to seek a GNE, which is defined as follows:
\begin{definition}\label{definition-gne}
    A GNE is a collective decision $x^* = \col\{x^*_i\}_{i\in\mathcal{I}}$, such that for all $i\in\mathcal{I}$,
    \begin{equation}\label{problem1}
        x_i^*\in\argmin_{x_i\in\mathbb{R}^{n_i}}\; J_i(x_i,x^*_{-i}),\quad\subjectto\; (x_i,x^*_{-i})\in\mathcal{X}.
    \end{equation}
\end{definition}
We make the following standing assumptions:
\begin{assumption}\label{assumption-obj-and-constraint}
\noindent
\begin{enumerate}[label=(\roman*)]
    \item For each $i\in\mathcal{I}$, $J_i$ is continuously differentiable and  convex in $x_i$, given $x_{-i}$;\label{assumption1-i}
    \item The constraint set $\mathcal{X}$ satisfies Slater's condition;\label{assumption1-ii} 
    \item The constraint matrix $G$ is of full row rank.\label{assumption1-iii}
\end{enumerate}
\end{assumption}
\begin{remark}
    We mention that Assumption \ref{assumption-obj-and-constraint}\ref{assumption1-i} and \ref{assumption1-ii} are standard assumptions in the GNE literature \cite{bianchi2021,facchinei2009,pavel2019}. Assumption \ref{assumption-obj-and-constraint}\ref{assumption1-iii} naturally holds in applications \cite{hall2025limits,ma2011decentralized,wang2014generalized}. As a matter of fact, in the context of distributed algorithms with safety considerations, an even stronger condition where each $g_i$ (recall that $G = \begin{bmatrix}
    g_1&\cdots&g_N
    \end{bmatrix}$) is of full row rank is commonly imposed \cite{liu2026achieving,mestres2023distributed,tan2025continuous}.
\end{remark}
Under Assumption \ref{assumption-obj-and-constraint}, we can characterize a GNE with a necessary and sufficient condition, known as the Karush-Kuhn-Tucker (KKT) condition \cite{facchinei2009}:
\begin{equation}\label{KKT}
    \begin{aligned}
        0_{n_i} &= \nabla_{x_i}J_i(x_i^*,x^*_{-i})+g_i^\top\lambda_i^*,\\
        0_m&\leq \lambda_i^*\perp(H-Gx^*)\geq 0_m,
    \end{aligned}
\end{equation}
where $\lambda_i^*\in\mathbb{R}^m$ is the optimal dual variable of the game problem (\ref{problem-formulation}). Let 
$$F(x) = \col\{\nabla_{x_i}J_i(x_i,x_{-i})\}_{i\in\mathcal{I}}$$
be the pseudo-gradient mapping, and $G_d = \diag\{g_i\}_{i\in\mathcal{I}}$, $\lambda^* = \col\{\lambda^*_i\}_{i\in\mathcal{I}}$, then we put the KKT condition (\ref{KKT}) as follows:
\begin{subequations}\label{KKT-vGNE}
    \begin{align}
        0 &= F(x^*)+G_d^\top\lambda^*,\label{KKT-vGNE-primal}\\
        0 &= \lambda^{*\top}[I_N\otimes(Gx^*-H)],\label{KKT-vGNE-complementary}\\
        0 &\geq Gx^*-H,\label{KKT-vGNE-primalfeasible}\\
        0 &\leq \lambda^*.\label{KKT-vGNE-dualfeasible}
    \end{align}
\end{subequations}
We focus on a refinement of GNE, the one that coincides with the solution of VI$(\mathcal{X},F)$, namely, the variational GNE (v-GNE), which requires the optimal dual variables of all players to be the same, i.e., there exists $\underline{\lambda}\in\mathbb{R}^m_+$ such that $\lambda_i^*=\underline{\lambda}$ for all $i\in\mathcal{I}$ \cite{facchinei2009}. 

\begin{remark}
    The concept of v-GNE usually indicates the attainment of a social welfare optimum under a fair scenario. For example, in a peer-to-peer electricity market \cite{Cadre2020}, the dual variable associated with each player is interpreted as the price with which the player trades the resource. The v-GNE of the market implies that all players have reached their optimal benefit under a uniform price system. Many other applications can also be suitably formulated as v-GNE problems, such as autonomous aerial vehicle coordination \cite{Hua2025}, post-disaster humanitarian relief \cite{Nagurney2016}, and transit systems \cite{Zhou2005}, to just name a few.
\end{remark}
For the existence and uniqueness of the v-GNE, we assume the following \cite{bianchi2021,peng2019}:
\begin{assumption}\label{assumption-pseudo-grad-mapping}
    The pseudo-gradient mapping $F(\cdot)$ has the following properties:
    \begin{enumerate}[label=(\roman*)]
        \item $\mu$-strongly monotone: there exists $\mu>0$ such that 
        $$(F(x) - F(y))^\top(x - y)\geq \mu\|x-y\|^2,\quad\forall x,y\in\mathbb{R}^n;$$
        \item $\theta$-Lipschitz continuous: there exists $\theta>0$ such that
        \[\|F(x)-F(y)\|\leq \theta\|x-y\|,\quad\forall x,y\in\mathbb{R}^n.\]
    \end{enumerate}
\end{assumption}
The objective of this paper is to design a safe GNE seeking dynamics or control input, such that the decision variables of the players converge to the exact v-GNE of the game, and the evolution of the dynamics is maintained within the constraint set $\mathcal{X}$. The design problem is formulated as follows:
\begin{problem}\label{remark-three-objectives}
    Design a dynamics or control input that satisfies the following:
    \begin{enumerate}[label=(\text{O\arabic*})]
        \item The decision variable $x$ remains within the constraint set $\mathcal{X}$ throughout the evolution of the dynamics, hence $\mathcal{X}$ is forward invariant;\label{objective-sf}
        \item The equilibrium of the dynamics recovers the KKT conditions \eqref{KKT-vGNE} with all players' dual variables being identical, hence is a v-GNE of the game problem (\ref{problem-formulation});\label{objective-eq}
        \item The dynamics asymptotically converges to its equilibrium.\label{objective-as}
    \end{enumerate}
\end{problem}
\section{Control barrier function-based GNE seeking}\label{sec:single-int}

In this section, we consider that all players are modeled as single-integrators
\[\dot{x}_i = u_i,\quad \forall i\in\mathcal{I}.\]
The goal is to design $u_i$, such that the objectives in Problem~\ref{remark-three-objectives} are satisfied.

\subsection{Distributed algorithm design}\label{sec:single-int-design}

We first adopt a full-information setting where each player has full information of the others' decisions $x_{-i}$, so that each player can accurately evaluate the partial derivative $\nabla_{x_i}J_i(x_i,x_{-i})$. Besides, each player can access the local information $J_i$, $g_i$, $h_i$. In order for the dual variables $\lambda_i\in\mathbb{R}^m$ to reach consensus, we allow information exchange among players through a communication network $\mathcal{G}(\mathcal{I},\mathcal{E},\mathcal{A})$, and we introduce an auxiliary variable $z_i\in\mathbb{R}^{m}$ to enforce the consensus. Each player $i$ communicates with neighbor $j\in\mathcal{N}_i$ the local information $z_i$ and $\lambda_i$, where the underlying communication network satisfies the following assumption:
\begin{assumption}\label{assumption-network}
    The communication network $\mathcal{G}(\mathcal{I},\mathcal{E},\mathcal{A})$ is undirected and connected.
\end{assumption}
We propose the following dynamics for each player $i\in\mathcal{I}$:
\begin{subequations}\label{dynamics}
    \begin{align}
        \dot{x}_i &= -\nabla_{x_i}J_i(x_i,x_{-i}) - g_i^\top \lambda_i, \label{dynamics-xi} \\
        \dot{z}_i &= \sum_{j\in\mathcal{N}_i} a_{ij}(\lambda_i - \lambda_j), \label{dynamics-zi} \\
        \lambda_i &= \begin{aligned}[t]
            \arg\min_{\lambda_i\geq 0} \; & \frac{1}{2} \left\| g_i^\top \lambda_i \right\|^2 + \lambda_i^\top \Biggr[ g_i \nabla_{x_i} J_i(x_i, x_{-i}) \\
            & - \alpha (g_i x_i - h_i) + \sum_{j\in\mathcal{N}_i} a_{ij}(z_i - z_j) \\
            & + \frac{1}{2} \sum_{j\in\mathcal{N}_i} a_{ij} (\lambda_i - 2\lambda_j) \Biggr],
        \end{aligned} \label{dynamics-u}
    \end{align}
\end{subequations}
where $\alpha>0$ is a constant parameter. 
\begin{remark}
Dynamics \eqref{dynamics} is based on primal-dual dynamics. We design \eqref{dynamics-xi} as a primal gradient descent of the following Lagrangian:
\[\mathscr{L}_i(x_i,\lambda_i) = J_i(x_i,x_{-i}) + \lambda_i^\top\sum_{i\in\mathcal{I}}(g_ix_i-h_i).\]
To enforce the identical dual variable requirement from the v-GNE definition, we integrate the consensus error of $\lambda_i$ as in \eqref{dynamics-zi}, whose asymptotic stability about the origin ensures consensus of the dual variables. By treating \eqref{dynamics-xi} as an affine control system with $\lambda_i$ being the control input and $\lambda_i=0$ as a nominal control, we design \eqref{dynamics-u} as a quadratic programming (QP)-based synthesis law that minimizes the control deviation from the nominal input, while restricting the dual variables within the following CBF-based admissible set:
\begin{multline}\label{admissible-set}
    K_\alpha(x) = \Big\{ \lambda \in \mathbb{R}^{Nm}_+ \;\Big|\; G\left(-F(x) - G_d^\top \lambda\right) \\
    \leq -\alpha (Gx - H) \Big\},
\end{multline}
where $\lambda = \col\{\lambda_i\}_{i\in\mathcal{I}}$ is the collective dual variable. The design of the admissible set $K_\alpha(x)$ is motivated by Nagumo's theorem \cite{khalil2002}, which states that for a set to be invariant, the vector field at any boundary point must direct inside or be tangent to the set. In our case, for the set $\mathcal{X}$ to be invariant, this requires the inequality $G(-F(x) - G_d^\top\lambda)\leq 0$ for all $x$ such that $Gx=H$. Whereas the set $K_\alpha(x)$ generalizes the notion to all points inside or on the boundary of the constraint set, by perturbing the right hand side of the inequality with the term $-\alpha(Gx-H)$ indicating the satisfaction of the constraint, where the parameter $\alpha>0$ adjusts the conservativeness of the perturbation. By Lemma~\ref{lemma-zcbf}, as long as $\lambda\in K_\alpha(x)$, the forward invariance of the constraint set $\mathcal{X}$ would be guaranteed. 
\end{remark}

To simplify the analysis of the dynamics and ensure the uniqueness of decision trajectories, we provide a compact representation of \eqref{dynamics} and establish its local Lipschitzness in the following lemma:
\begin{lemma}\label{lemma-compact-lipschitz}
    Dynamics \eqref{dynamics} can be written in the following compact form:
    \begin{subequations}\label{dynamics-stack}
    \begin{align}
        \dot{x} &= -F(x) - G_d^\top \lambda, \label{dynamics-stack-x} \\
        \dot{z} &= L\lambda, \label{dynamics-stack-z} \\
        \lambda &= \begin{aligned}[t]
            \arg\min_{\lambda \geq 0} \; & \frac{1}{2}\|G_d^\top \lambda\|^2 + \frac{1}{2}\lambda^\top L\lambda \\
            & + \lambda^\top \Big[ G_d F(x) - \alpha(G_dx - H_s) + Lz \Big],
        \end{aligned} \label{dynamics-stack-u}
    \end{align}
    \end{subequations}
where $z = \col\{z_i\}_{i\in\mathcal{I}}$. Moreover, under Assumption~\ref{assumption-obj-and-constraint}, the right-hand side of \eqref{dynamics-stack} is locally Lipschitz with respect to $(x,z)$.
\end{lemma}
\begin{proof}
    It is straightforward to rewrite \eqref{dynamics-xi} and \eqref{dynamics-zi} in the compact form of \eqref{dynamics-stack-x} and \eqref{dynamics-stack-z}. For \eqref{dynamics-u}, we notice that it is essentially a QP problem, and its constraint $\lambda_i\geq 0$ satisfies Slater's condition, therefore, the following KKT condition for \eqref{dynamics-u} implies its optimality:
\begin{subequations}\label{u-KKT}
    \begin{align}
        \begin{split}
            0 ={}& g_i g_i^\top \lambda_i + g_i \nabla_{x_i} J_i(x_i, x_{-i}) - \alpha (g_i x_i - h_i) \\
            & + \sum_{j\in\mathcal{N}_i} a_{ij}(z_i - z_j) + \sum_{j\in\mathcal{N}_i} a_{ij}(\lambda_i - \lambda_j) - w_i,
        \end{split}\label{u-KKT-1}\\
        0 ={}& w_i^\top \lambda_i,\label{u-KKT-2}\\
        0 \leq{}& w_i,\label{u-KKT-3}\\
        0 \leq{}& \lambda_i.
    \end{align}
\end{subequations}
where $w_i\in\mathbb{R}^{m}$ is the dual variable of problem \eqref{dynamics-u}. For notational conciseness, we let $w = \col\{w_i\}_{i\in\mathcal{I}}$, $H_s = \col\{h_i\}_{i\in\mathcal{I}}$, $L = \mathcal{L}\otimes I_m$, and $G_d$ be as defined in \eqref{KKT-vGNE-primal}. Then \eqref{u-KKT} can be written in the following compact form:
\begin{subequations}\label{u-KKT-stack}
    \begin{align}
        \begin{split}
            0 ={}& G_dG_d^\top \lambda + G_dF(x) - \alpha(G_dx - H_s) \\
            & + Lz + L\lambda - w,
        \end{split}\label{u-KKT-stack-1}\\
        0 ={}& w^\top \lambda,\label{u-KKT-stack-2}\\
        0 \leq{}& w,\label{u-KKT-stack-3}\\
        0 \leq{}& \lambda.
    \end{align}
\end{subequations}
Notice that the equivalence of \eqref{u-KKT-2} and \eqref{u-KKT-stack-2} follows from the fact that for non-negative vectors $w_i$ and $\lambda_i$, $w^\top \lambda = 0$ if and only if $w_i^\top \lambda_i=0$ for all $i$. Moreover, we observe that \eqref{u-KKT-stack} is in fact the KKT condition for \eqref{dynamics-stack-u}. Hence, dynamics \eqref{dynamics} can be put in the compact representation of \eqref{dynamics-stack}. 

To show the Lipschitzness of \eqref{dynamics-stack}, we first notice that \eqref{dynamics-stack-u} is a QP with positive definite quadratic term and admits a unique and finite solution, since $\ker(\mathcal{L}) = \text{span}(\mathds{1}_N)$, then if $\|G_d^\top\lambda\|^2 = 0$ and $\lambda^\top L\lambda = 0$ with the same nonzero $\lambda$, there exists $\lambda'\in\mathbb{R}^m$ such that $\lambda_i = \lambda'$ and $g_i^\top \lambda' = 0,\;\forall i$, which contradicts Assumption~\ref{assumption-obj-and-constraint}\ref{assumption1-iii}. Consider the solution to \eqref{dynamics-stack-u} as a function $\Lambda(x,z):\mathbb{R}^n\times \mathbb{R}^{Nm}\to\mathbb{R}^{Nm}$ of $x$ and $z$, then by \cite[Theorem 2.1]{coroianu2016}, we have that $\Lambda(x,z)$ is locally Lipschitz with respect to $(x,z)$, which further implies the local Lipschitzness of the dynamics \eqref{dynamics-stack}. 
\end{proof}
\begin{remark}
    Note that decoupling the game problem \eqref{problem-formulation} via the notion of constraint mismatch variable as in \cite{mestres2023distributed,tan2025continuous} would be difficult here, as the objectives of the players are interdependent. 
    Furthermore, the feasibility of the dual variable synthesis and the Lipschitzness of the closed-loop dynamics would hold if every $g_i$ is of full row rank \cite[Lemma 4.5]{mestres2023distributed}. In contrast, these properties are guaranteed in our design under the milder condition that the constraint matrix $G$ is of full row rank, i.e., Assumption~\ref{assumption-obj-and-constraint}\ref{assumption1-iii}.
\end{remark}
\subsection{Convergence analysis}\label{sec:full-info-convergence}
In this subsection, we analyze the convergence of dynamics \eqref{dynamics-stack}. Specifically, we demonstrate that the objectives \ref{objective-sf}--\ref{objective-as} are satisfied. First, we show the forward invariance of the constraint set $\mathcal{X}$, which is \ref{objective-sf}.
\begin{lemma}\label{lemma-forwardinvariance}
    Consider dynamics \eqref{dynamics-stack}. Under Assumption~\ref{assumption-network}, for any feasible initial condition $x(0)\in\mathcal{X}$ and $\forall \alpha>0$, the variable $\lambda$ satisfies $\lambda\in K_\alpha(x)$, and thus, the feasible set $\mathcal{X}$ is forward invariant.
\end{lemma}
\begin{proof}
    Consider the admissible set \eqref{admissible-set} and notice that the KKT condition \eqref{u-KKT-stack-1} and \eqref{u-KKT-stack-3} imply
    $$G_dG_d^\top \lambda + G_dF(x)-\alpha(G_dx-H_s)+Lz+L\lambda\geq0.$$
    Since $G_d = \diag\{g_i\}_{i\in\mathcal{I}}$ and $H_s = \col\{h_i\}_{i\in\mathcal{I}}$, it holds that $G = (\mathds{1}_N^\top\otimes I_{m})G_d$ and $H = (\mathds{1}_N^\top\otimes I_{m})H_s$. Left-multiplying the above inequality with $\mathds{1}_N^\top\otimes I_{m}$ gives
    $$GG_d^\top \lambda+GF(x)-\alpha(Gx-H)\geq0,$$
    which means $\lambda\in K_\alpha(x)$. Moreover, by Lemma~\ref{lemma-compact-lipschitz}, the right-hand side of dynamics \eqref{dynamics-stack-x}--\eqref{dynamics-stack-z} is locally Lipschitz. Therefore, by Lemma~\ref{lemma-zcbf}, the feasible set $\mathcal{X}$ is forward invariant under \eqref{dynamics-stack}.
\end{proof}
\begin{remark}
    To enforce invariance of the constraint set $\mathcal{X}$ by CBF, Definition~\ref{definition-vcbf} requires that the admissible set $K_\alpha(x)$ need to be non-empty. Lemma~\ref{lemma-forwardinvariance} ensures this by showing that the QP \eqref{dynamics-stack-u} admits a solution and the solution stays within $K_\alpha(x)$.
\end{remark}
Next, we verify that the equilibrium of \eqref{dynamics-stack} is exactly the v-GNE of the game problem \eqref{problem-formulation}, which shows \ref{objective-eq}.
\begin{lemma}\label{lemma-vgneeq}
    Under Assumption~\ref{assumption-network}, the equilibrium $(\bar{x},\bar{z})$ of dynamics \eqref{dynamics-stack} along with the corresponding dual variable $\bar{\lambda} = \Lambda(\bar{x},\bar{z})$ solves the KKT conditions \eqref{KKT-vGNE}, and $\bar{\lambda} = \mathds{1}_N\otimes \underline{\lambda}$, for some $\underline{\lambda}\in\mathbb{R}^m_+$. Hence, $\bar{x}$ is a v-GNE of the game problem \eqref{problem-formulation}.
\end{lemma}
\begin{proof}
    The equilibrium point $(\bar{x},\bar{z})$ of \eqref{dynamics-stack} satisfies the following:
    \begin{subequations}\label{equilibrium condition}
        \begin{align}
            F(\bar{x})+G_d^\top \bar{\lambda}&=0,\label{equilibrium-1}\\
            L\bar{\lambda}&=0,\label{equilibrium-2}\\
            \text{KKT for \eqref{dynamics-stack-u} at }&(\bar{x},\bar{z}).\label{equilibrium-3}
        \end{align}
    \end{subequations}
    Trivially, \eqref{equilibrium-1} implies \eqref{KKT-vGNE-primal}, the constraint of \eqref{dynamics-stack-u} implies \eqref{KKT-vGNE-dualfeasible}, \eqref{equilibrium-2} under the assumption that the communication network is connected implies that $\bar{\lambda} = \mathds{1}_N\otimes \underline{\lambda}$, for some $\underline{\lambda}\in\mathbb{R}^m_+$, and \eqref{KKT-vGNE-primalfeasible} is guaranteed by the forward invariance of $\mathcal{X}$ established in Lemma~\ref{lemma-forwardinvariance}. Next we show that \eqref{equilibrium-3} implies \eqref{KKT-vGNE-complementary}. We write out \eqref{equilibrium-3} as follows:
    \begin{subequations}
    \begin{align}
        \begin{split}
            0 ={}& G_dG_d^\top \bar{\lambda} + G_dF(\bar{x}) - \alpha(G_d\bar{x} - H_s) \\
            & + L\bar{z} + L\bar{\lambda} - \bar{w},
        \end{split}\label{u-KKT-equilibrium-1}\\
        0 ={}& \bar{w}^\top \bar{\lambda},\\
        0 \leq{}& \bar{w},\\
        0 \leq{}& \bar{\lambda}.
    \end{align}
\end{subequations}
    By \eqref{equilibrium-1} and Assumption~\ref{assumption-network}, the first two equations above reduce to
    \[-\alpha\bar{\lambda}^\top(G_d\bar{x}-H_s)=0,\]
    and further
    $$\underline{\lambda}^{\top}(\mathds{1}_N^\top \otimes I_m)(G_d\bar{x} - H_s) = \underline{\lambda}^{\top}(G\bar{x}-H)=0,$$
    which implies \eqref{KKT-vGNE-complementary}. Therefore, the equilibrium condition \eqref{equilibrium condition} solves the KKT condition \eqref{KKT-vGNE}, and the equilibrium of \eqref{dynamics-stack} is a v-GNE of the game problem \eqref{problem-formulation}.
\end{proof}
Having established \ref{objective-sf} and \ref{objective-eq}, we prove the asymptotic stability of \eqref{dynamics-stack}, i.e., \ref{objective-as}, thereby completing our analysis as summarized in the following theorem:
\begin{theorem}\label{theorem-singleintegrator}
    Consider dynamics \eqref{dynamics-stack}. Under Assumptions~\ref{assumption-obj-and-constraint}, \ref{assumption-pseudo-grad-mapping} and \ref{assumption-network}, if $\alpha>\frac{\theta^2}{4\mu}$, then for any feasible initial condition $x(0)\in\mathcal{X}$, the decision variable $x$ asymptotically converges to the v-GNE of the game problem \eqref{problem-formulation}, with forward invariance of the constraint set $\mathcal{X}$.
\end{theorem}
\begin{proof}
    We consider a Lyapunov function candidate 
    $$V = \alpha V_1+V_2,$$
    where
    $$V_1 = \frac{1}{2}\|x-\bar{x}\|^2\quad\text{and}\quad V_2 = \frac{1}{2}\|z-\bar{z}\|^2.$$  
    Taking the derivative of $V_2$ along \eqref{dynamics-stack-z} gives
\begin{align*}
    \dot{V}_2 ={}& (z-\bar{z})^\top L\lambda \\
    ={}& (\lambda-\bar{\lambda})^\top L(z-\bar{z}) \\
    ={}& (\lambda-\bar{\lambda})^\top \Big[ (w-\bar{w}) - L(\lambda-\bar{\lambda}) + \alpha G_d(x-\bar{x}) \\
    & - G_d(F(x)-F(\bar{x})) - G_d G_d^\top (\lambda-\bar{\lambda}) \Big],
\end{align*}
where the last equality holds by subtracting \eqref{u-KKT-stack-1} with \eqref{u-KKT-equilibrium-1}. Taking the derivative of $\alpha V_1$ along \eqref{dynamics-stack-x} gives
\begin{align*}
    \dot{V}_1 &= -(x-\bar{x})^\top(F(x) + G_d^\top \lambda)\\
    & = -(x-\bar{x})^\top\left[(F(x)+G_d^\top\lambda) - (F(\bar{x})+G_d^\top\bar{\lambda})\right]\\
    & = -(x-\bar{x})^\top(F(x)-F(\bar{x})) - (x-\bar{x})^\top G_d^\top(\lambda-\bar{\lambda})\\
    &\leq -\mu\|x-\bar{x}\|^2-(x-\bar{x})^\top G_d^\top(\lambda-\bar{\lambda}),
\end{align*}
where the second equality holds by \eqref{equilibrium-1} and the last inequality holds by the strong monotonicity of the pseudo-gradient mapping $F(\cdot)$. Then we compute the derivative of $V$ as follows:
\begin{align*}
    \dot{V} =& \alpha\dot{V}_1+\dot{V}_2\\
     \leq& -\mu\alpha\|x-\bar{x}\|^2 + (\lambda-\bar{\lambda})^\top(w-\bar{w})-\lambda^\top L\lambda \\
    &- \|G_d^\top(\lambda-\bar{\lambda})\|^2- (\lambda-\bar{\lambda})^\top G_d(F(x)-F(\bar{x})).
\end{align*}
By Young's inequality, and using $\lambda^\top w=0$ and $\bar{\lambda}^\top\bar{w}=0$, we further obtain
\begin{align*}
    \dot{V} \leq{}& -\mu\alpha\|x-\bar{x}\|^2 + (\lambda-\bar{\lambda})^\top(w-\bar{w}) - \lambda^\top L\lambda \\
    & - \|G_d^\top(\lambda-\bar{\lambda})\|^2 + \|G_d^\top(\lambda-\bar{\lambda})\|^2 \\
    & + \frac{1}{4}\|F(x)-F(\bar{x})\|^2 \\
    ={}& -\mu\alpha\|x-\bar{x}\|^2 - \lambda^\top\bar{w} - \bar{\lambda}^\top w - \lambda^\top L\lambda \\
    & + \frac{1}{4}\|F(x)-F(\bar{x})\|^2 \\
    \leq{}& -\left(\mu\alpha - \frac{\theta^2}{4}\right)\|x-\bar{x}\|^2 - \lambda^\top\bar{w} - \bar{\lambda}^\top w - \lambda^\top L\lambda.
\end{align*}
Since $\lambda,w,\bar{\lambda},\bar{w}\geq0$, then $\dot{V}\leq 0$, as long as $\alpha>\frac{\theta^2}{4\mu}$. By the radial unboundedness of $V$, we define $\mathcal{P}$ to be any compact sublevel set of $V$ that contains the initial condition $(x(0),z(0))$. Moreover, by Nagumo's theorem \cite{khalil2002}, we know that $\mathcal{P}$ is forward invariant. Let $E = \{(x,z)\in\mathcal{P} \mid \dot{V} = 0\}\subseteq \{(x,z)\in\mathcal{P} \mid x = \bar{x}, L\bar{\lambda} = L\Lambda(\bar{x},\bar{z}) = 0\}$. By LaSalle's invariance principle, $(x(t),z(t))$ converges asymptotically to the largest positively invariant subset of $E$. Therefore, $x(t)\to\bar{x}$ asymptotically, and by Lemmas~\ref{lemma-forwardinvariance} and \ref{lemma-vgneeq}, we conclude that the dynamics \eqref{dynamics-stack} solves the game problem \eqref{problem-formulation} for v-GNE with safety guarantees.
\end{proof}
\begin{remark}
    We discuss the role of the parameter $\alpha$ and provide an intuition for the condition $\alpha>\frac{\theta^2}{4\mu}$ in Theorem~\ref{theorem-singleintegrator}. We observe that for points lying in the interior of the constraint set $\mathcal{X}$ and a relatively small $\alpha$, by Nagumo's theorem, the admissible directions of the trajectory either point inside of the set or parallel to the boundary, whereas for a larger $\alpha$, more admissible directions (possibly pointing outside of $\mathcal{X}$) are available. 
    In this light, $\alpha$ serves as a measure for the richness of the admissible directions. Furthermore, by \cite[Proposition 1 \& Proposition 5]{gadjov2022exact}, Assumption~\ref{assumption-pseudo-grad-mapping} implies $F(\cdot)$ being $\frac{\theta^2}{\mu}$-Lipschitz, where the value of $\frac{\theta^2}{\mu}$ indicates how drastically the pseudo-gradient can vary. Thus, the condition $\alpha>\frac{\theta^2}{4\mu}$ means that the available admissible directions for the seeking trajectory need to be sufficiently rich to handle the variation of the pseudo-gradient, representing an interplay of equilibrium seeking and safety enhancing.  
\end{remark}
\section{Safe GNE seeking under partial-information setting}\label{sec:partialinfo}

In this section, we extend our safe GNE seeking design to partial-information settings, where each player can only access the decisions of its neighbors through the communication network $\mathcal{G}(\mathcal{I},\mathcal{E},\mathcal{A})$, rather than a full decision profile of all other players (cf. \eqref{dynamics}). 

Similar to \cite{gadjov2022exact,pavel2019}, we adopt an estimation scheme where each player $i$ maintains estimates of all other players' decisions. Specifically, we augment player $i$'s decision vector by $\mathbf{x}_i = \col\{\mathbf{x}_i^1, \dots, \mathbf{x}_i^{i-1}, x_i, \mathbf{x}_i^{i+1}, \dots, \mathbf{x}_i^N\}$, where $\mathbf{x}_i^j$ denotes its estimate of player $j$'s action. By the definition of the augmented decision variable, the true decision $x_i$ can be recovered by left-multiplying \textbf{x}$_i$ with a selection matrix, i.e., $x_i = \mathcal{R}_i\textbf{x}_i$, where
\[\mathcal{R}_i = \begin{bmatrix}
    0_{n_i\times n_{<i}} & I_{n_i} & 0_{n_i\times n_{>i}}
\end{bmatrix},\]
$n_{<i} = \sum_{j = 1}^{i-1}n_j$ and $n_{>i} = \sum_{j=i+1}^Nn_j$. Moreover, we notice that when the augmented decision variables $\textbf{x}_i$ reach consensus, i.e., $\textbf{x}_i = \textbf{x}_j$ for all $i,j\in\mathcal{I}$, each player's estimates of other players' decisions are the true decisions. In this way, the partial derivative $\nabla_{x_i}J_i(x_i,x_{-i})$ can be approximated by computing $\nabla_{x_i}J_i(\textbf{x}_i)$. 

However, prior to the convergence of the algorithm, the estimation error might drive the decision variables outside the constraint set. To enhance the safety, we consider reformulating the constraint set of the game by incorporating the decision estimates. By letting $\textbf{x} = \col\{\textbf{x}_i\}_{i\in\mathcal{I}}$, $\mathcal{R} = \diag\{\mathcal{R}_i\}_{i\in\mathcal{I}}$, and considering that $x = \mathcal{R}\textbf{x}$, the constraint set $\mathcal{X} = \{x\in\mathbb{R}^n|Gx\leq H\}$ can be rewritten, in terms of the augmented decision variable, as
\begin{equation}\label{augmented-constraints}
    \hat{\mathcal{X}} = \{\textbf{x}\in\mathbb{R}^{Nn}|G\mathcal{R}\textbf{x}\leq H\}.
\end{equation}

Under partial-information setting, the design objectives in Problem~\ref{remark-three-objectives} boil down to the design of dynamics for the augmented decision variable $\textbf{x}_i$, such that
\begin{enumerate}[label=(O\arabic*$'$)]
    \item The constraint set $\hat{\mathcal{X}}$ is forward invariant;\label{objective-partial-sf}
    \item At equilibrium, the augmented decision variables of all players are identical, i.e., $\bar{\textbf{x}}_i = \bar{\textbf{x}}_j$, $\forall i,j\in\mathcal{I}$, and the true decisions $\bar{x}_i = \mathcal{R}_i\bar{\textbf{x}}_i$ constitute the v-GNE;\label{objective-partial-eq}
    \item The dynamics asymptotically converges to its equilibrium.\label{objective-partial-as}
\end{enumerate}

\subsection{Distributed algorithm design}
Based on the discussion above, we design the dynamics for each player $i\in\mathcal{I}$ as follows:
\begin{subequations}\label{dynamics-partial-info}
    \begin{align}
        \dot{\mathbf{x}}_i &= -\mathcal{R}_i^\top \nabla_{x_i}J_i(\mathbf{x}_i) - \mathcal{R}_i^\top g_i^\top\lambda_i - c\sum_{j\in\mathcal{N}_i} a_{ij}(\mathbf{x}_i - \mathbf{x}_j), \label{dynamics-partial-x} \\
        \dot{z}_i &= \sum_{j\in\mathcal{N}_i} a_{ij}(\lambda_i - \lambda_j), \label{dynamics-partial-z} \\
        \lambda_i &= \arg\min_{\lambda_i\geq 0} \; \frac{1}{2}\|g_i^\top\lambda_i\|^2 + \lambda_i^\top \Big[ g_i\nabla_{x_i}J_i(\mathbf{x}_i) \notag \\
        &\quad + cg_i\mathcal{R}_i\sum_{j\in\mathcal{N}_i} a_{ij}(\mathbf{x}_i - \mathbf{x}_j) - \alpha(g_i\mathcal{R}_i\mathbf{x}_i - h_i) \notag \\
        &\quad + \sum_{j\in\mathcal{N}_i} a_{ij}(z_i-z_j) + \frac{1}{2}\sum_{j\in\mathcal{N}_i} a_{ij}(\lambda_i - 2\lambda_j) \Big], \label{dynamics-partial-lambda}
    \end{align}
\end{subequations}
where $c>0$ is a design parameter. Let $\textbf{F}(\textbf{x}) = \col\{\nabla_{x_i}J_i(\textbf{x}_i)\}_{i\in\mathcal{I}}$ be the extended pseudo-gradient and $\textbf{L} = \mathcal{L}\otimes I_n$ (cf. $L$ defined in \eqref{u-KKT-stack}), then by following a similar approach as in Section~\ref{sec:single-int-design}, we can rewrite \eqref{dynamics-partial-info} as follows:
\begin{subequations}\label{dynamics-partial-stack}
    \begin{align}
    \dot{\textbf{x}} &= -\mathcal{R}^\top \textbf{F}(\textbf{x}) - \mathcal{R}^\top G_d^\top\lambda - c\textbf{Lx},\\
    \dot{z} &= L\lambda,\\
    \lambda &= \argmin_{\lambda\geq 0} 
    \begin{multlined}[t]
    \frac{1}{2}\|G_d^\top\lambda\|^2 + \lambda^\top [ G_d\textbf{F}(\textbf{x}) + cG_d\mathcal{R}\textbf{Lx} \\- \alpha(G_d\mathcal{R}\textbf{x} - H_s) + Lz ] + \frac{1}{2}\lambda^\top L\lambda.
    \end{multlined}\label{dyanmics-partial-lambda-stack}
    \end{align}
\end{subequations}
We note that the design for \eqref{dynamics-partial-x} can be decomposed into a consensus-based estimation, and a primal descent-like update, perturbed by the estimation error, for the true decision variable:
\begin{align*}
    \dot{\textbf{x}}_i^{-i} &= -c\sum_{j\in\mathcal{N}_i}a_{ij}(\textbf{x}_i^{-i} - \textbf{x}_j^{-i}),\\
    \dot{x}_i &= -\nabla_{x_i}J_i(\textbf{x}_i) - g_i^\top\lambda_i - c\sum_{j\in\mathcal{N}_i}a_{ij}(x_i - \textbf{x}_j^i),
\end{align*}
where $\textbf{x}_i^{-i} = \col\{\textbf{x}_i^k\}_{k\in\mathcal{I}\backslash i}$ and $\textbf{x}_j^{-i}$ is defined similarly. Analogous to the full-information case, \eqref{dynamics-partial-z} is designed to enforce consensus of the dual variables and the dual variable synthesis law \eqref{dynamics-partial-lambda} is designed to keep the control deviation away from the nominal input $\lambda=0$ as small as possible, and restrict the dual variables within the following CBF-based admissible set:
\begin{align*}
    \hat{K}_\alpha(\mathbf{x}) = \Big\{ \lambda \in \mathbb{R}_+^{Nm} \;\Big|\;& G\mathcal{R}\big( -\mathcal{R}^\top \mathbf{F}(\mathbf{x}) - \mathcal{R}^\top G_d^\top\lambda \\
    & - c\mathbf{Lx} \big) \leq -\alpha(G\mathcal{R}\mathbf{x} - H) \Big\}.
\end{align*}
\begin{remark}\label{remark-partial-compare}
    We compare the partial-information dynamics \eqref{dynamics-partial-stack} with the full-information one \eqref{dynamics-stack}. By letting $\textbf{F}_a(\textbf{x}) = \mathcal{R}^\top\textbf{F}(\textbf{x}) + c\textbf{L}\textbf{x}$ be the augmented pseudo-gradient, and noticing that $\mathcal{R}\mathcal{R}^\top = I_n$, we observe that the form of dynamics \eqref{dynamics-partial-stack} and \eqref{dynamics-stack} are identical. Recall that dynamics \eqref{dynamics-stack} essentially solves the VI$(\mathcal{X},F)$ in a distributed manner, where for each player, the global decision information is available and the partial derivative $\nabla_{x_i}J_i(x_i,x_{-i})$ can be readily computed. Similarly, \eqref{dynamics-partial-stack} can be regarded as a dynamics that solves VI$(\hat{\mathcal{X}},\textbf{F}_a)$, whereas for each player, the calculation of the augmented partial derivative $-\mathcal{R}_i^\top \nabla_{x_i}J_i(\textbf{x}_i) - c\sum_{j\in\mathcal{N}_i} a_{ij}(\textbf{x}_i - \textbf{x}_j)$ only requires the augmented decisions of its neighbors and itself. In other words, the partial- and full-information settings are equivalent after augmentation of the decision variable and the pseudo-gradient.
\end{remark}

\subsection{Convergence analysis}
In this subsection, we show that the objectives \ref{objective-partial-sf} -- \ref{objective-partial-as} are satisfied. By virtue of \eqref{remark-partial-compare}, the analysis in Section~\ref{sec:full-info-convergence} can be naturally applied here as long as $\textbf{F}_a(\cdot)$ is Lipschitz continuous and strongly monotone, which are established by the following two lemmas:
\begin{lemma}
    Under Assumption~\ref{assumption-pseudo-grad-mapping}, the extended pseudo-gradient mapping $\textbf{F}(\cdot)$ is $\theta_1$-Lipschitz, for some $\theta_1\in[\mu,\theta]$. Furthermore, the augmented pseudo-gradient mapping $\textbf{F}_a(\cdot)$ is $\theta_2$-Lipschitz, where $\theta_2 = \theta_1 + c\sigma_{\max}(\mathcal{L})$.
\end{lemma}
\begin{proof}
    By Assumption~\ref{assumption-pseudo-grad-mapping}, and following from \cite[Lemma 3]{bianchi2021}, we have that the extended pseudo-gradient $\textbf{F}(\cdot)$ is $\theta_1$-Lipschitz for some $\theta_1\in[\mu,\theta]$. Then $\forall \textbf{x},\textbf{y}\in\mathbb{R}^{Nn}$
    \begin{align*}
        \|\textbf{F}_a(\textbf{x}) - \textbf{F}_a(\textbf{y})\| &\leq \|\mathcal{R}^\top (\textbf{F}(\textbf{x}) - \textbf{F}(\textbf{y}))\| + c\|\textbf{L}(\textbf{x} - \textbf{y})\|\\
        &\leq \|\textbf{F}(\textbf{x}) - \textbf{F}(\textbf{y})\| + c\|\textbf{L}\|\cdot\|\textbf{x} - \textbf{y}\|\\
        &\leq (\theta_1 + c\sigma_{\max}(\mathcal{L}))\|\textbf{x} - \textbf{y}\|,
    \end{align*}
    i.e., $\textbf{F}_a(\cdot)$ is $\theta_2$-Lipschitz with $\theta_2 = \theta_1 + c\sigma_{\max}(\mathcal{L})$.
\end{proof}

\begin{lemma}\cite[Lemma 3]{pavel2019}\label{lemma:restricted-strongly-monotone}
    Under Assumption \ref{assumption-obj-and-constraint}, \ref{assumption-pseudo-grad-mapping} and \ref{assumption-network}, let 
    \[\Psi = \begin{bmatrix} 
        \frac{\mu}{N} & -\frac{\theta + \theta_1}{2\sqrt{N}} \\ 
        -\frac{\theta + \theta_1}{2\sqrt{N}} & c\sigma_2(\mathcal{L}) - \theta_1 
    \end{bmatrix},\]
    where $\sigma_2(\mathcal{L})$ is the second smallest eigenvalue of the matrix $\mathcal{L}$. Then $\Psi\succ 0$ for any $c>c_{\min} = \frac{1}{\sigma_2(\mathcal{L})}\left(\frac{(\theta + \theta_1)^2}{4\mu} + \theta_1\right)$, and
    \begin{equation}\label{restricted-monotonicity-inequality}
        (\textbf{x} - \textbf{y})^\top(\textbf{F}_a(\textbf{x}) - \textbf{F}_a(\textbf{y}))\geq\bar{\mu}\|\textbf{x} - \textbf{y}\|^2,
    \end{equation}
    where $\bar{\mu} = \sigma_{\min}(\Psi)$ and $\textbf{y} = \mathds{1}_N\otimes y$ for some $x\in\mathbb{R}^n$.
\end{lemma}

\begin{remark}
    The property that $\textbf{F}_a(\cdot)$ possesses according to Lemma~\ref{lemma:restricted-strongly-monotone} is the restricted strong monotonicity (cf. strong monotonicity in Assumption~\ref{assumption-pseudo-grad-mapping}), where the second variable $\textbf{y}$ is restricted to the consensus subspace. This property is ensured by choosing the design parameter $c$ large enough. We emphasize that it is possible for $\textbf{F}_a(\cdot)$ to be restricted strongly monotone without the parameter $c$, i.e., when $c=1$. However, this will impose an additional requirement on the communication network, as in \cite{gadjov2018passivity,weijian2024}, and our design avoids this topological restriction by shifting the burden to an additional design parameter. Moreover, we notice that in Section~\ref{sec:full-info-convergence}, the strong monotonicity inequality is used only in Theorem~\ref{theorem-singleintegrator}, where the second variable of the inequality is the equilibrium of the collective decision. Therefore, the stability analysis for the full-information dynamics can be applied here if the equilibrium of the stacked augmented decision variable $\bar{\textbf{x}}$ lies in the consensus subspace.
\end{remark}

We intend to adapt the analysis from Section~\ref{sec:full-info-convergence} to partial-information settings. However, by Remark~\ref{remark-partial-compare}, dynamics \eqref{dynamics-partial-stack} solves VI$(\hat{\mathcal{X}},\textbf{F}_a)$ with safety guarantees rather than the original VI$(\mathcal{X},F)$, and how the former implies the latter needs to be illustrated. We first show the forward invariance of the constraint set $\mathcal{X}$ and the augmented set $\hat{\mathcal{X}}$, which is \ref{objective-partial-sf}.

\begin{lemma}\label{lemma-partial-forwardinvariance}
    Consider dynamics \eqref{dynamics-partial-stack}. Under Assumption~\ref{assumption-network}, for any $\alpha>0$ and any feasible initial condition $x(0)\in\mathcal{X}$ or equivalently, $\textbf{x}(0)\in\hat{\mathcal{X}}$, both the sets $\hat{\mathcal{X}}$ and $\mathcal{X}$ are forward invariant.
\end{lemma}
\begin{proof}
    Notice that dynamics \eqref{dynamics-partial-stack} is formally identical to dynamics \eqref{dynamics-stack}, where the decision variable $x$ is replaced by the augmented decision $\textbf{x}$ and the constraint set $\mathcal{X}$ is replaced by $\hat{\mathcal{X}}$. Then by a result similar to Lemma~\ref{lemma-forwardinvariance}, but applicable to dynamics \eqref{dynamics-partial-stack}, the constraint set $\hat{\mathcal{X}}$ is forward invariant. Since $x = \mathcal{R}\mathbf{x}$, by \eqref{augmented-constraints}, it immediately follows that the original constraint set $\mathcal{X}$ is also forward invariant.
\end{proof}

Next, we verify that the equilibrium of \eqref{dynamics-partial-stack} is the v-GNE of the game, where the collective augmented decision variable lies in the consensus subspace, and thus \ref{objective-partial-eq} is satisfied.

\begin{lemma}\label{lemma:partial-eq}
    Under Assumption~\ref{assumption-network}, the equilibrium of the collective augmented decision variable $\bar{\textbf{x}}$ lies in the consensus subspace, i.e., $\bar{\textbf{x}} = \mathds{1}_N\otimes \bar{x}$ for some $\bar{x}\in\mathbb{R}^n$. Moreover, the equilibrium $(\bar{\textbf{x}},\bar{z})$ of system (\ref{dynamics-partial-stack}) along with the corresponding dual variable $\bar{\lambda} = \Lambda(\bar{\textbf{x}},\bar{z})$ solves the KKT conditions (\ref{KKT-vGNE}), and $\bar{\lambda} = \mathds{1}_N\otimes \underline{\lambda}$, for some $\underline{\lambda}\in\mathbb{R}^m_+$. Hence, $\bar{x} = \mathcal{R}\bar{\textbf{x}}$ is a v-GNE of the game problem (\ref{problem-formulation}).
\end{lemma}
\begin{proof}
    Notice that at the equilibrium, the following holds:
    \begin{equation}\label{lemma-partial-eq-proof-equation1}
        -\mathcal{R}^\top \textbf{F}(\bar{\textbf{x}}) - \mathcal{R}^\top G_d^\top\bar{\lambda} - c\textbf{L}\bar{\textbf{x}} = 0.
    \end{equation}
    By left-multiplying $\mathds{1}^\top\otimes I_n$ on both sides of \eqref{lemma-partial-eq-proof-equation1}, and by noticing that $(\mathds{1}^\top\otimes I_n)\mathcal{R}^\top = I_n$, we have $\textbf{F}(\bar{\textbf{x}}) + G_d^\top\bar{\lambda} = 0$. Plugging this back to \eqref{lemma-partial-eq-proof-equation1} gives $\textbf{L}\bar{\textbf{x}}=0$, i.e., the equilibrium of the augmented decision variable lies in the consensus subspace and each player's estimates of other players' decisions are true decisions. Therefore, $F(\bar{x}) + G_d^\top\bar{\lambda}=0$, which implies \eqref{KKT-vGNE-primal}. The rest of the proof follows from Lemma~\ref{lemma-vgneeq}.
\end{proof}
In addition to \ref{objective-partial-sf} and \ref{objective-partial-eq}, we show that \eqref{dynamics-partial-stack} is asymptotically stable, which satisfies \ref{objective-partial-as}, and thus safe GNE seeking under partial-information setting is achieved.
\begin{theorem}\label{theorem-partial-singleintegrator}
    Consider dynamics (\ref{dynamics-partial-stack}). Under Assumptions \ref{assumption-obj-and-constraint}, \ref{assumption-pseudo-grad-mapping} and \ref{assumption-network}, if $\alpha> \frac{(\theta_1 + c\sigma_{\max}(\mathcal{L}))^2}{4\bar{\mu}}$, then for any feasible initial condition $x(0)\in\mathcal{X}$, or equivalently $\textbf{x}(0)\in\hat{\mathcal{X}}$, the true decision variable $x$ asymptotically converges to the v-GNE of the game problem (\ref{problem-formulation}) with forward invariance of the constraint set $\mathcal{X}$.
\end{theorem}
\begin{proof}
    Consider a Lyapunov function candidate
    \[V = \frac{\alpha}{2}\|\textbf{x} - \bar{\textbf{x}}\|^2 + \frac{1}{2}\|z-\bar{z}\|^2,\]
    and follow a similar procedure as in Theorem~\ref{theorem-singleintegrator}, where the restricted monotonicity inequality \eqref{restricted-monotonicity-inequality} is used whenever the strong monotonicity inequality is encountered, since the equilibrium $\bar{\textbf{x}}$ lies in the consensus subspace. Then we can obtain
    \begin{multline*}
    \dot{V} \leq -\left( \alpha\bar{\mu} - \frac{(\theta_1 + c\sigma_{\max}(\mathcal{L}))^2}{4} \right) \|\mathbf{x} - \bar{\mathbf{x}}\|^2 \\
    - \lambda^\top \bar{w} - \bar{\lambda}^\top w - (\lambda - \bar{\lambda})^\top L(\lambda - \bar{\lambda})
\end{multline*}
    and $\dot{V}\leq 0$ if 
    \[\alpha> \frac{(\theta_1 + c\sigma_{\max}(\mathcal{L}))^2}{4\bar{\mu}}.\]
    By following a LaSalle invariance principle argument as in Theorem~\ref{theorem-singleintegrator}, $\textbf{x}$ converges to $\bar{\textbf{x}}$ asymptotically, which implies that $x = \mathcal{R}\textbf{x}$ converges to $\bar{x}=\mathcal{R}\bar{\textbf{x}}$, i.e., the dynamics (\ref{dynamics-partial-stack}) solves the game problem (\ref{problem-formulation}) for v-GNE. Moreover, by Lemma~\ref{lemma-partial-forwardinvariance}, we conclude that safe GNE seeking under partial-information setting is achieved.
\end{proof}
\begin{remark}
    We notice that the lower bound for the parameter $\alpha$ differs from the one in Theorem~\ref{theorem-singleintegrator}. This arises from the introduction of consensus-based estimation dynamics, and consequently, the communication network, through the graph Laplacian, affects the conditioning of the augmented pseudo-gradient mapping and hence the sufficient lower bound on $\alpha$. To have a closer look at how the connectivity of the graph affects the lower bound of $\alpha$, we let $c = \frac{\kappa}{\sigma_2(\mathcal{L})}$ with $\kappa>\theta_1+\frac{(\theta+\theta_1)^2}{4\mu}$, so that $\Psi\succ 0$ is independent of the graph Laplacian, and the lower bound for $\alpha$ becomes $(\theta_1 + \kappa\frac{\sigma_{\max}(\mathcal{L})}{\sigma_2(\mathcal{L})})/4\bar{\mu}$, which is a function of only the spectral ratio $\frac{\sigma_{\max}(\mathcal{L})}{\sigma_2(\mathcal{L})}$. The spectral ratio is proportional to $\frac{\Delta(\mathcal{L})+1}{\delta(\mathcal{L})}$ by \cite{goldberg2006bounding}, where $\Delta(\mathcal{L})$ and $\delta(\mathcal{L})$ are the maximum and minimum degree of the graph, and therefore, the more densely and evenly the players are connected through the communication network, the lower the bound is for $\alpha$.
\end{remark}
\section{Equilibrium seeking with multi-integrator agents}\label{sec:multi-int}

In this section, we extend the CBF-based design to GNE seeking for players modelled as multi-integrators
\begin{align}\label{dynamics-multi-int}
    x_i^{(r)} = u_i.
\end{align}
Here $x_i = x_i^{(0)}\in\mathbb{R}^{n_i}$ is the state of the agent and is also considered as the decision variable of the game, $u_i\in\mathbb{R}^{n_i}$ is the input and $r\in\mathbb{N}_+$ is the highest derivative order of the agents. Practical examples of multi-integrator dynamics can be found in \cite{kant2026tracking,romero2024two}. We aim to design the input $u_i$ in a distributed way such that the decision variable $x_i$ reaches the v-GNE. Let $x = \col\{x_i\}_{i\in\mathcal{I}}$, $x^k = \col\{x_i^{(k)}\}_{i\in\mathcal{I}}$, $\forall k = \{1,2,\cdots,r-1\}$ and $u = \col\{u_i\}_{i\in\mathcal{I}}$, then the multi-integrator multi-agent system can be cast into the following compact form:
\begin{equation}\label{multi-int-compact}
    \dot{x} = x^1, \quad\dot{x}^1 = x^2,\quad\cdots,\quad\dot{x}^{r-1} = u.
\end{equation}
For the GNE seeking problem with multi-integrator agents, besides the GNE characterization in \eqref{problem1}, we require that at steady-state, the higher-order derivatives of the decision variables are identically zero, i.e., $x_i^{(1)} = x_i^{(2)} = \cdots = x_i^{(r-1)} = 0$. Then the GNE of the game is a decision profile $x^{*}$, such that for all $i\in\mathcal{I}$,
\begin{equation}\label{problem-multi-int}
    \begin{aligned}
        x^*_i\in\argmin&\quad J_i\left(x_i+\sum_{k=1}^{r-1}\delta_i^k x_i^{(k)}, x_{-i}^*+\sum_{k=1}^{r-1}\delta_{-i}^kx_{-i}^{(k)}\right)\\
        \subjectto&\quad(x_i,x_{-i}^*)\in\mathcal{X},\\ &\quad x_i^{(1)} = x_i^{(2)} = \cdots = x_i^{(r-1)} = 0,
    \end{aligned}
\end{equation}
with appropriately sized matrices $\delta_i^k$ and $\delta_{-i}^k$. It is worth noticing that when $x_i^{(k)}=0$, $\forall k$, the above formulation reduces to Definition \ref{definition-gne}. Therefore, we design the algorithm with respect to the combination of $x_i$ and its higher-order derivatives, rather than $x_i$ alone, to solve the GNE seeking problem, while maintaining the evolution of the agents within the coupled constraint set.

\subsection{Coordinate transformation}
Building upon the new formulation \eqref{problem-multi-int} of the GNE, we simplify the problem by adopting a coordinate transformation approach that transforms $x_i$ to a combination of $x_i$ and its higher-order derivatives. We perform the transformation from a CBF perspective, so that the multi-integrator reduces to a single integrator and the invariance of $\mathcal{X}$ translates to the invariance of new sets under the new coordinate, paving the way for the adaptation of dynamics \eqref{dynamics}.

The key of implementing the CBF technique is to constrain the input of the control system within an admissible set defined as in (\ref{zcbf-soln}). Recalling \eqref{multi-int-compact}, to guarantee the forward invariance of $\mathcal{X}$, by temporarily treating $x^1$ as an input, we can design the admissible set in terms of $x^1$ as follows:
\[K^1_{\alpha_1}(x) = \{x^1\in\mathbb{R}^n\mid G(x^1+\alpha_1 x)\leq \alpha_1 H\},\]
and introduce the following coordinate transformation:
\begin{align*}
    \left\{
  \begin{array}{r @{} l}
    \zeta^1     &= x + \frac{1}{\alpha_1}x^1, \\[0.5em]
    \tilde{u}^1 &= x^1 + \frac{1}{\alpha_1}x^2.
  \end{array}
\right.
\end{align*}
Notice that $K^1_{\alpha_1}(x)$ is equivalent to the set $\mathcal{Z}_1 = \{\zeta^1\in\mathbb{R}^n\mid G\zeta^1\leq H\}$. We carry on the same process with the following series of the coordinate transformations for $k = \{1,2,\cdots,r-1\}$, $\zeta^0 = x$, $\tilde{u}^0 = x^1$, and $\alpha_k>0$:
\begin{align}\label{coordinate-transform-multi-int}
    \left\{
  \begin{array}{r @{} l}
    \zeta^k     &= \zeta^{k-1} + \frac{1}{\alpha_k}\tilde{u}^{k-1}, \\[0.5em]
    \tilde{u}^k &= \frac{d}{dt}\zeta^k,
  \end{array}
\right.
\end{align}
where $\zeta^k = \col\{\zeta^k_i\}_{i\in\mathcal{I}}$ and $\tilde{u}^k = \col\{\tilde{u}^k_{i}\}_{i\in\mathcal{I}}$. Then we obtain a series of constraint sets $\mathcal{Z}_k$ on $\zeta^k$ and admissible sets $K^k_{\alpha_k}(\zeta^{k-1})$ for $\tilde{u}^{k-1}$ as follows:
\begin{equation}\label{constraint-new-coordinate}
    \begin{aligned}
        \mathcal{Z}_k &= \{\zeta^k\in\mathbb{R}^n|G\zeta^k\leq H\},\\
        K^k_{\alpha_k}(\zeta^{k-1}) &= \{\tilde{u}^{k-1}\in\mathbb{R}^n\mid G(\tilde{u}^{k-1}+\alpha_k\zeta^{k-1})\leq \alpha_kH\}.
    \end{aligned}
\end{equation}
Note that $\zeta^k$ is essentially a combination of the decision variable $x$ and its derivatives up to $k$th order, i.e., 
\begin{equation}\label{coordinate-transformation-zetak}
    \zeta^k = x + \sum_{j=1}^{k}\delta^jx^j,
\end{equation}
where $\delta^j$ are the descending coefficients, except the first one, of the polynomial $\prod_{\ell=1}^{k}(s + \frac{1}{\alpha_\ell})$.

\begin{remark}\label{remark-four-sets-relationships}
    It is worthwhile to clarify the relationships of the sets $\mathcal{X}$ and $K_{\alpha_k}^k$. Since $K_{\alpha_r}^{r}$ is the admissible set built upon the constraint set $\mathcal{Z}_{r-1}$, by Lemma \ref{lemma-zcbf}, any input $\tilde{u}^{r-1}\in K_{\alpha_r}^{r}$ would render $\zeta^{r-1}\in\mathcal{Z}_{r-1}$ for all future time. Moreover, $\mathcal{Z}_{r-1}$ and $K_{\alpha_{r-1}}^{r-1}$ are equivalent through the coordinate transformation (\ref{coordinate-transform-multi-int}), meaning that $\tilde{u}^{r-1}\in K_{\alpha_r}^{r}$ would render the invariance of $K_{\alpha_{r-1}}^{r-1}$. By carrying forward such relationship, we conclude that any input $\tilde{u}^{r-1}\in K_{\alpha_r}^{r}$ would eventually lead to the forward invariance of the constraint set $\mathcal{X}$ (a schematic illustration is in Fig. \ref{fig:admissible-sets-relationships}).
    \begin{figure}[htbp]
    \centering
    \[\mathcal{X} \xLeftarrow[\text{invariance}]{\text{implies}}K_{\alpha_1}^1 \xLeftarrow[\text{invariance}]{\text{implies}}\cdots\xLeftarrow[\text{invariance}]{\text{implies}} K_{\alpha_r}^{r}\]
    \caption{The relationships of the admissible sets $K_{\alpha_k}^k$ and $\mathcal{X}$}
    \label{fig:admissible-sets-relationships}
\end{figure}
\end{remark}
\begin{remark}
    We notice that our method of coordinate transformation resembles the notion of high-order control barrier functions (HOCBF) in \cite{xiao2021high}, where the admissible control set is designed by constraining the Lie derivative of a scalar HOCBF at the order of the system's relative degree. In contrast, our method can be regarded as a vectorized generalization of the HOCBF. Moreover, as an extension of the single-agent control design proposed in HOCBF literature \cite{tan2021high,xiao2021high}, we introduce a distributed safe control design approach.
\end{remark}

\subsection{Distributed algorithm design}
For player $i\in\mathcal{I}$, consider the $(r-1)$th coordinate transformation, i.e., \eqref{coordinate-transformation-zetak} when $k=r-1$, as follows:
\begin{align}\label{multi-integrator-coordinate-transformation}
    \left\{
  \begin{array}{r @{} l}
    \zeta^{r-1}_i     &= x_i + \sum_{j=1}^{r-1} \delta^j x^{(j)}_i, \\[1.2em]
    \tilde{u}^{r-1}_i &= x_i^{(1)} + \sum_{j=1}^{r-2} \delta^j x_i^{(j+1)} + \delta^{r-1} u_i.
  \end{array}
\right.
\end{align}
Then the multi-integrator \eqref{dynamics-multi-int} in the new coordinates reads
\begin{subequations}
    \begin{align}
    \dot{\zeta}_i^{r-1} &= \tilde{u}_i^{r-1},\label{mul-integrator-zeta}\\
    \dot{x}_i^{d} &= A^{d}x_i^{d}+B^{d}\tilde{u}_i^{r-1},\\
    x_i &= \zeta_i^{r-1} -\Delta_i x_i^d,
    \end{align}
\end{subequations}
where $x_i^{d} = \col\{x_i^{(k)}\}_{k\in\{1,\cdots,r-1\}}$,
\[A^d = \begin{bmatrix}
    0&I_{n_i}&0&\cdots&0\\
    0&0&I_{n_i}&\cdots&0\\
    \vdots&\vdots&\vdots&\ddots&\vdots\\
    0&0&0&\cdots&I_{n_i}\\
    -\frac{1}{\delta^{r-1}}I_{n_i}&-\frac{\delta^1}{\delta^{r-1}}I_{n_i}&-\frac{\delta^2}{\delta^{r-1}}I_{n_i}&\cdots&-\frac{\delta^{r-2}}{\delta^{r-1}}I_{n_i}
\end{bmatrix},\]
$B^d = \begin{bmatrix}
    0&\cdots&0&\frac{1}{\delta^{r-1}}I_{n_i}
\end{bmatrix}^\top$ and $\Delta = \begin{bmatrix}
\delta^1 I_{n_i} & \cdots & \delta^{r-1}I_{n_i}
\end{bmatrix}$. Furthermore, the definition of the GNE in \eqref{problem-multi-int} can be reformulated under the new coordinate as a collective decision profile $(\zeta^{r-1})^* = \col\{(\zeta_i^{r-1})^*\}_{i\in\mathcal{I}}$ such that for all $i\in\mathcal{I}$:
\begin{equation}\label{problem-reformulation-multi-int}
    \begin{aligned}
        \zeta_i^*\in\argmin&\quad J_i(\zeta_i^{r-1},(\zeta^{r-1}_{-i})^*)\\
        \subjectto&\quad (\zeta_i^{r-1},(\zeta_{-i}^{r-1})^*)\in\mathcal{Z}_{r-1}, \\
        &\quad x_i^{(k)}=0,\;\forall k\in\{1,\cdots,r-1\},
    \end{aligned}
\end{equation}
where the constraint set $\mathcal{Z}_{r-1}$ is defined in \eqref{constraint-new-coordinate}. Based on the reformulation of GNE in \eqref{problem-reformulation-multi-int}, where the constraint set $\mathcal{Z}_{r-1}$ is identical to the original set $\mathcal{X}$ in both form and invariance properties (by Remark~\ref{remark-four-sets-relationships}), and noticing that dynamics (\ref{mul-integrator-zeta}) is a single integrator, we use algorithm (\ref{dynamics}) for (\ref{mul-integrator-zeta}), and obtain the following closed-loop system:
\begin{subequations}\label{algorithm-muti-int}
    \begin{align}
        \dot{\zeta}^{r-1}_i &= \tilde{u}^{r-1}_i, \label{algorithm-multi-int-zeta} \\
        \tilde{u}^{r-1}_i &= -\nabla_{x_i} J_i(\zeta^{r-1}_i,\zeta^{r-1}_{-i}) - g_i^\top\lambda_i, \\
        \dot{z}_i &= \sum_{j\in\mathcal{N}_i} a_{ij}(\lambda_i-\lambda_j), \\
        \lambda_i &= \begin{aligned}[t]
            \arg\min_{\lambda_i\geq0} \; & \frac{1}{2}\left\|g_i^\top \lambda_i\right\|^2 + \lambda_i^\top \Big[ g_i\nabla_{x_i} J_i(\zeta^{r-1}_i,\zeta^{r-1}_{-i}) \\
            & - \alpha_r (g_i\zeta^{r-1}_i-h_i) + \sum_{j\in\mathcal{N}_i}a_{ij}(z_i-z_j) \\
            & + \frac{1}{2}\sum_{j\in\mathcal{N}_i}a_{ij}(\lambda_i-2\lambda_j) \Big],
        \end{aligned} \label{algorithm-multi-int-lambda} \\
        \dot{x}_i^{d} &= A^{d}x_i^{d}+B^{d}\tilde{u}_i^{r-1}, \\
        x_i &= \zeta_i^{r-1} -\Delta_i x_i^d.
    \end{align}
\end{subequations}
Let $\tilde{u}^{r-1} = \col\{\tilde{u}^{r-1}_i\}_{i\in\mathcal{I}}$, $x^d = \col\{x_i^d\}_{i\in\mathcal{I}}$ and $\Delta = \diag\{\Delta_i\}_{i\in\mathcal{I}}$. Then we rewrite the above closed-loop system in the following compact form:
\begin{subequations}\label{multi-integrator-dynamics-clsdlp-stacked}
    \begin{align}
        \dot{\zeta}^{r-1} =& \tilde{u}^{r-1},\label{multi-integrator-dynamics-zeta}\\
        \tilde{u}^{r-1} =& -F(\zeta^{r-1}) - G_d^\top\lambda,\\
        \dot{z} =& L\lambda,\\
        \lambda =& \begin{aligned}[t]
            \argmin_{\lambda\geq0}\;&\frac{1}{2}\|G_d^\top \lambda\|^2+\lambda^\top\left[G_dF(\zeta^{r-1})\right.\\&-\alpha_r(G_d \zeta^{r-1}-H_s)+Lz]+\frac{1}{2}\lambda^\top L\lambda,
        \end{aligned}
        \label{multi-integrator-lambda}\\
        \dot{x}^d =& (I_N\otimes A^d)x^d + (I_N\otimes B^d)\tilde{u}^{r-1},\label{multi-integrator-dynamic-xd}\\
        x =& \; \zeta - \Delta\cdot x^d.
    \end{align}
\end{subequations}
\begin{remark}
    In the case of $r=1$, i.e., the players are single integrators, the coordinate transformation is merely $\zeta^{0} = x$, dynamics \eqref{multi-integrator-dynamic-xd} is removed, and \eqref{multi-integrator-dynamics-clsdlp-stacked} reduces to \eqref{dynamics-stack}. The dynamics for $r=2$ was presented in \cite{Meng2026}, and can also be regarded as a special case of \eqref{multi-integrator-dynamics-clsdlp-stacked}.
\end{remark}

\subsection{Convergence analysis}
Considering the coordinate transformation we applied, as well as the fact that the equivalence of the two formulations of the GNE under two coordinates is based on the vanishment of $x^{(k)}_i$, the design objectives in Problem~\ref{remark-three-objectives} can be broken down into the following:
\begin{enumerate}[label=(O\arabic*$''$)]
    \item The constraint set $\mathcal{Z}_{r-1}$ is forward invariant;\label{objective-multi-int-sf}
    \item The equilibrium of (\ref{multi-integrator-dynamics-zeta}) is a v-GNE defined by (\ref{problem-reformulation-multi-int});\label{objective-multi-int-eq}
    \item The single-integrator (\ref{multi-integrator-dynamics-zeta}) asymptotically converges to its equilibrium and the sub-system (\ref{multi-integrator-dynamic-xd}) is asymptotically stable.\label{objective-multi-int-as}
\end{enumerate}
The following theorem shows the convergence of the proposed algorithm:
\begin{theorem}
    Consider dynamics (\ref{multi-integrator-dynamics-clsdlp-stacked}). Under Assumptions \ref{assumption-obj-and-constraint}, \ref{assumption-pseudo-grad-mapping}, and \ref{assumption-network}, for any feasible initial condition $x(0)\in\mathcal{X}$, the following hold:
    \begin{enumerate}[label=(\roman*)]
        \item the variable $\zeta^{r-1}$ asymptotically converges to the v-GNE defined by (\ref{problem-reformulation-multi-int}) with $\mathcal{Z}_{r-1}$ being forward invariant, provided $\alpha_r>\frac{\theta^2}{4\mu}$;
        \item the dynamics (\ref{multi-integrator-dynamic-xd}) is input-to-state stable (ISS) with respect to the input $\tilde{u}^{r-1}$, provided $\alpha_k>0$, for all $k \in \{1,2,\cdots,r-1\}$.
    \end{enumerate}
    Hence, (\ref{multi-integrator-dynamics-clsdlp-stacked}) solves the game problem (\ref{problem-formulation}) for the v-GNE with forward invariance of the constraint set $\mathcal{X}$.
\end{theorem}
\begin{proof}
    (i) As dynamics (\ref{multi-integrator-dynamics-zeta})--(\ref{multi-integrator-lambda}) are formally identical to dynamics \eqref{dynamics-stack} with $x$ and $\mathcal{X}$ replaced by $\zeta^{r-1}$ and $\mathcal{Z}_{r-1}$, respectively. Then, by arguments similar to Lemmas~\ref{lemma-forwardinvariance}, \ref{lemma-vgneeq} and Theorem~\ref{theorem-singleintegrator}, but applicable to dynamics (\ref{multi-integrator-dynamics-zeta})--(\ref{multi-integrator-lambda}), we have that if $\alpha_r>\frac{\theta^2}{4\mu}$, $\zeta^{r-1}$ asymptotically converges to the v-GNE defined by (\ref{problem-reformulation-multi-int}) with $\mathcal{Z}_{r-1}$ being forward invariant, thereby objectives \ref{objective-multi-int-sf} and \ref{objective-multi-int-eq} are satisfied;

    (ii) Since $\alpha_k>0$, for all $k \in \{1,2,\cdots,r-1\}$, $A^d$ is Hurwitz with the eigenvalues $\{-\alpha_1,-\alpha_2,\cdots,-\alpha_{r-1}\}$, and $\tilde{u}^{r-1}$ approaches $0$ asymptotically. Then (\ref{multi-integrator-dynamic-xd}) is ISS and $x^d$ converges to $0$ asymptotically.

    Given (i) and (ii), objective \ref{objective-multi-int-as} is satisfied. Since the GNE defined by (\ref{problem-reformulation-multi-int}) and (\ref{problem1}) are identical when $x_i^{(k)}=0$, for all $k\in\{1,2,\cdots,r-1\}$, and $x^d(t)$ is vanishing by (ii), the equilibrium $\bar{x}$ is the v-GNE of the game problem (\ref{problem-formulation}). Moreover, by Remark~\ref{remark-four-sets-relationships}, the forward invariance of $\mathcal{Z}_{r-1}$ implies the forward invariance of $\mathcal{X}$. Therefore, (\ref{multi-integrator-dynamics-clsdlp-stacked}) solves the game problem (\ref{problem-formulation}) for the v-GNE with forward invariance of the constraint set $\mathcal{X}$. 
\end{proof}
\begin{remark}
    We note that distributed safe GNE seeking for multi-integrator agents under partial-information setting can be similarly carried out by keeping an estimate for variable $\zeta^{r-1}_{-i}$ as considered in \cite{bianchi2021}. Since the coordinate transformation reduces the multi-integrator dynamics into single-integrator form, the design from Section~\ref{sec:partialinfo} can be applied directly by augmenting $\zeta_i^{r-1}$ with the estimate for $\zeta_{-i}^{r-1}$. The convergence analysis of the algorithm follows directly from Section \ref{sec:partialinfo} and \ref{sec:multi-int} and is thus omitted for brevity.
\end{remark}
\section{Simulations}\label{sec:sim}

In this section, we apply the algorithm designed for single integrators and double integrators to two practical examples. The first example compares the CBF-based algorithm with a traditional primal-dual method and an interior-point method (IPM) to highlight the advantage of the designed algorithm. The second example verifies the applicability of the algorithm designed for multi-integrator agents.

\subsection{Networked Cournot Game}
We consider a Cournot game problem where $N=4$ companies compete in $m=3$ markets. Company $i\in\mathcal{I} = \{1,2,3,4\}$ provides $x_i\in\mathbb{R}^{n_i}$ amount of products to the market it participates in, and matrix $A_i\in\mathbb{R}^{m\times n_i}$ specifies the participation of company $i$ in market $j$ by letting $[A_i]_{jk} = 1$ if and only if company $i$ provides product $k$ to market $j$, and $[A_i]_{jk} = 0$ otherwise. Let $A = [A_1,\cdots,A_N]$ and $x = \col\{x_i\}_{i\in\mathcal{I}}$. Then $Ax$ is a vector whose entries document the amount of products provided to the corresponding markets. The goal of the companies is to minimize their own objectives, respectively, while ensuring that the total amounts of the products they provide do not exceed the capacities of the markets, which is described by a vector $r\in\mathbb{R}^3$ and can be separated as $r = \sum_{i=1}^{N}r_i$. The problem is then formulated as follows:
\begin{equation}\label{sim1-problem}
    \begin{aligned}
        \min_{x_i\in\mathbb{R}^{n_i}}&\quad c_i(x_i) - P^\top(Ax)A_ix_i\\
        \subjectto&\quad \sum_{i=1}^{4}A_i x_i\leq \sum_{i=1}^{4}r_i,
    \end{aligned}
\end{equation}
where $c_i = x_i^\top Q_i x_i + q_i^\top x_i$ is the cost, $P(Ax) = \bar{P} - \Xi Ax$ is the price, $\bar{P}\in\mathbb{R}^m$ and $\Xi \in\mathbb{R}^{m\times m}$. We choose the parameters (or the entries of the parameters) $Q_i,q_i,\bar{P},\Xi,r_i$ and $A_i$ randomly from the intervals $[1,4],(0,1),[100,150],(0,3),[10,20]$ and the set $\{0,1\}$, such that $Q_i\succ 0$ and $\Xi\succ 0$. Moreover, the 4 companies are allowed to exchange their dual variables through a connected undirected communication network described by Fig. \ref{fig:communication}. In order to indicate the forward invariance of the constraint set, we define $g(x) = Ax-r = \col\{g_i(x)\}_{i\in\mathcal{I}}$, and intuitively, $\max\{g_i(x)\}_{i\in\mathcal{I}}\leq0$ for all time means that the forward invariance is ensured.
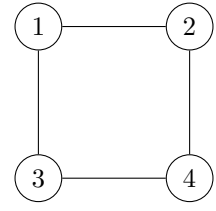
\begin{figure}
    \centering
    \begin{tikzpicture}[node distance=2cm]
    \node[circle, draw] (A) {1};
    \node[circle, draw, right of=A] (B) {2};
    \node[circle, draw, below of=A] (C) {3};
    \node[circle, draw, below of=B] (D) {4};

    \draw (A) -- (B);
    \draw (A) -- (C);
    \draw (B) -- (D);
    \draw (C) -- (D);
    \end{tikzpicture}
    \caption{Communication network}
    \label{fig:communication}
\end{figure}

\textit{A comparison with primal-dual method}: we implement the full information CBF-based algorithm (\ref{dynamics}), its partial-information counterpart (\ref{dynamics-partial-info}) and a traditional primal-dual algorithm (e.g. (5) in \cite{weijian2024}) for problem (\ref{sim1-problem}). The simulation results are presented in Fig. \ref{fig:sim1}: the left plot shows that all three algorithms are able to seek the v-GNE of problem (\ref{sim1-problem}); however, the right plot, where the shaded area indicates the infeasible region, i.e., $\exists\; i$ such that $g_i(x)>0$, shows that at the very beginning of the simulation, at least one decision variable by the primal-dual method jumps outside of the constraint before gradually converges back, while the decision variables by the two CBF-based algorithms are kept within the constraint set for all time.
\begin{figure}[htbp]
    \centering
    \includegraphics[width = 8cm]{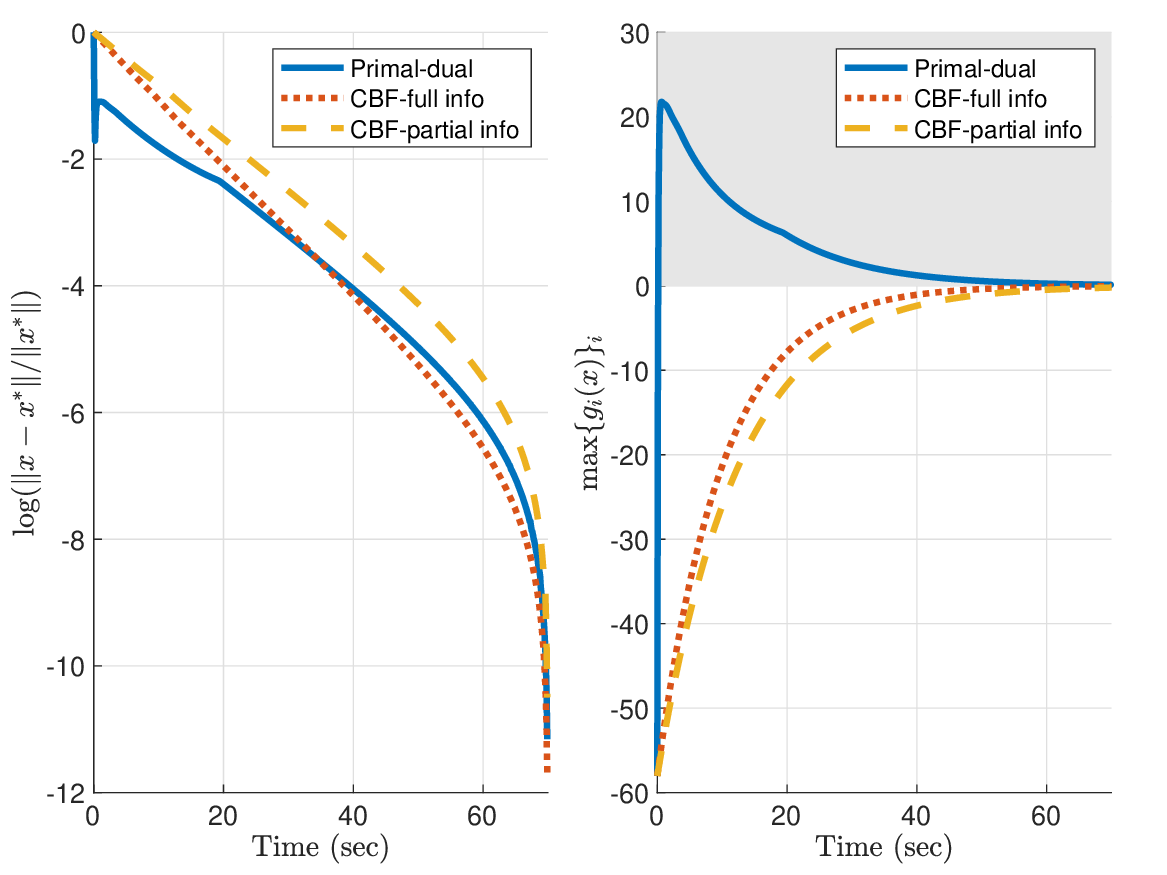}
    \caption{The trajectories of $\log(\|x-x^*\|/\|x^*\|)$ and $\max\{g_i(x)\}$ by the CBF-based algorithm and the primal-dual method.}
    \label{fig:sim1}
\end{figure}

\textit{A comparison with interior-point method}: we implement an interior-point algorithm with time-varying penalty parameter (e.g. (12) in \cite{romano2022}) for problem (\ref{sim1-problem}), and compare the results with the CBF-based algorithm. The simulation results are presented in Fig. \ref{fig:sim1.1}. For illustration, we initialize both algorithms at the exact v-GNE of the game. As shown in the plot, the trajectories by the interior-point algorithm transiently diverge from the exact v-GNE, due to the log-barrier penalty on the objectives, but they eventually return to the v-GNE, when the parameter of the penalty is tuned to the extent that the barrier becomes negligible. In contrast, the trajectories by the CBF-based algorithm stay at the v-GNE for all time.
\begin{figure}[htbp]
    \centering
    \includegraphics[width = 8cm]{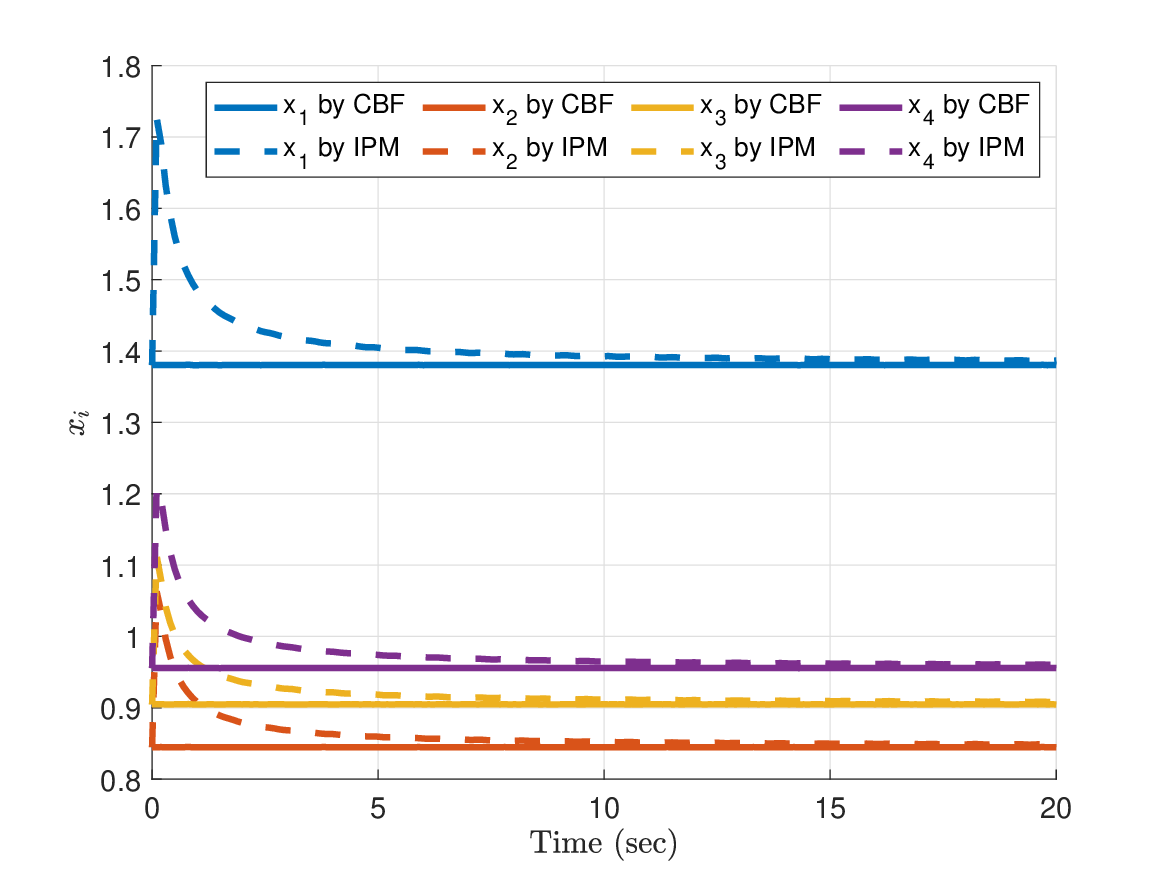}
    \caption{The trajectories of $x_i$ by the CBF-based algorithm and IPM.}
    \label{fig:sim1.1}
\end{figure}

\subsection{Rate Control over Wireless ad hoc Networks}
We consider an extremely simple wireless ad hoc network with 3 nodes and 2 links as in Fig. \ref{fig:adhoc}, and there are 4 users trying to transfer data through this network. The vector $G_i\in\mathbb{R}^{2}$ specifies the links user $i$ uses by letting $[G_i]_j = 1$ if and only if user $i$ uses link $j$ and $[G_i]_j=0$ otherwise. If user $i$ uses link $j$, then the user needs to decide its data rate $x_i\in\mathbb{R}$ by solving the following problem:
\begin{equation}\label{sim2-problem}
    \begin{aligned}
        \min_{x_i\in\mathbb{R}}&\quad -\chi_i\log(x_i+1)+D^\top(x)G_ix_i\\
        \subjectto&\quad Gx\leq H,
    \end{aligned}
\end{equation}
where $G = [G_1,G_2,G_3,G_4]\in\mathbb{R}^{2\times 4}$, $H = \col\{H_1,H_2\}\in\mathbb{R}^2$ stores the capacities of the two links and can be decoupled by $H = \sum_{i=1}^{4}h_i$, and $D(x) = \col\{d_1(x),d_2(x)\}\in\mathbb{R}^2$ with
\[d_j = \frac{k_j}{H_j - [Gx]_j + \xi_j}\]
describes the delays of the links. The process of solving for problem (\ref{sim2-problem}) is controlled by a double integrator, i.e., (\ref{dynamics-multi-int}) when $r=2$, and we denote the derivative of the state as $v_i$, i.e., $\dot{x}_i = v_i$. We mention that it is practical to implement the algorithm for double-integrator agents, since the first order-derivative can be interpreted as a measurement for the link backlogs \cite{Tian2005,Paganini2002}. The parameters (or the entries of the parameters) $H_j,\chi_i,k_j$ and $\xi_j$ are randomly drawn from the intervals $[10,16],[10,20],[10,30]$ and $[20,40]$. Moreover, the users are able to exchange their dual variables through the communication network in Fig. \ref{fig:communication}.
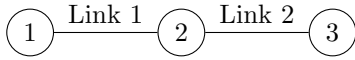
\begin{figure}
    \centering
    \begin{tikzpicture}[node distance=2cm]
    \node[circle, draw] (A) {1};
    \node[circle, draw, right of=A] (B) {2};
    \node[circle, draw, right of=B] (C) {3};

    \draw (A) -- node[above] {Link 1}(B);
    \draw (B) -- node[above] {Link 2}(C);
    \end{tikzpicture}
    \caption{Wireless ad hoc network}
    \label{fig:adhoc}
\end{figure}

We implement the full-information algorithm \eqref{algorithm-muti-int} and its partial-information counterpart to solve for problem (\ref{sim2-problem}). The simulation results are presented in Fig. \ref{fig:sim2} and \ref{fig:sim2-partial}. In both Fig. \ref{fig:full-a} and \ref{fig:partial-a}, we can see from the left plots that the data rates of the users eventually reach a steady-state,  from the right upper plots that the derivatives of the decision variables eventually vanish and from the right lower plots that the forward invariance of the feasible set is ensured. In both Fig. \ref{fig:full-b} and \ref{fig:partial-b}, we can see from the two upper plots that the dual variables of the users eventually converge to the same value, and from the lower plot that the complementary slackness is eventually satisfied, which verifies that the steady-state is the v-GNE of the problem. Moreover, each user's estimate of the collective vector $\zeta$ is plotted in Fig. \ref{fig:estimation}, which demonstrates that the estimates eventually reach a consensus, meaning that accurate estimation is achieved upon convergence.
\begin{figure}[htbp]
    \centering
    \begin{subfigure}[b]{0.45\textwidth}
        \centering
        \includegraphics[width = 5.8cm]{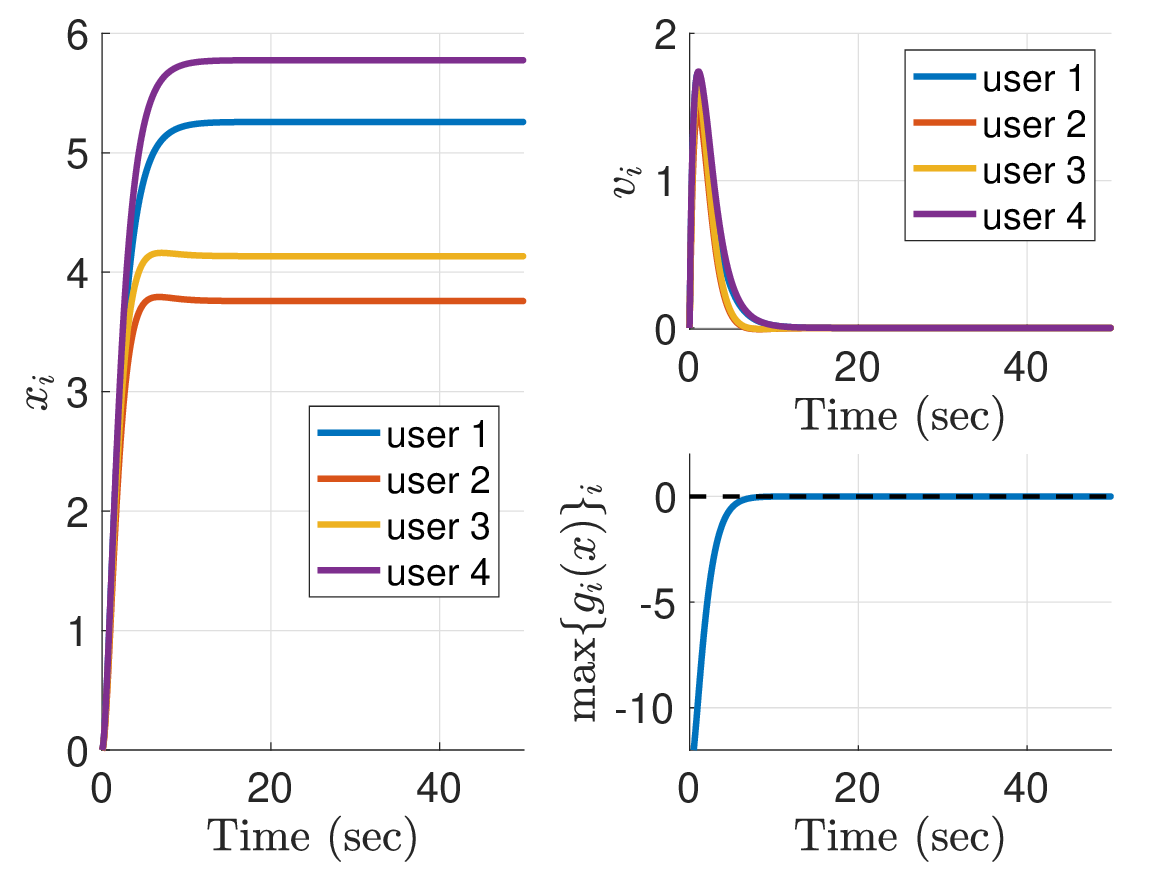}
        \caption{}
        \label{fig:full-a}
    \end{subfigure}
    \begin{subfigure}[b]{0.45\textwidth}
        \centering
        \includegraphics[width = 5.8cm]{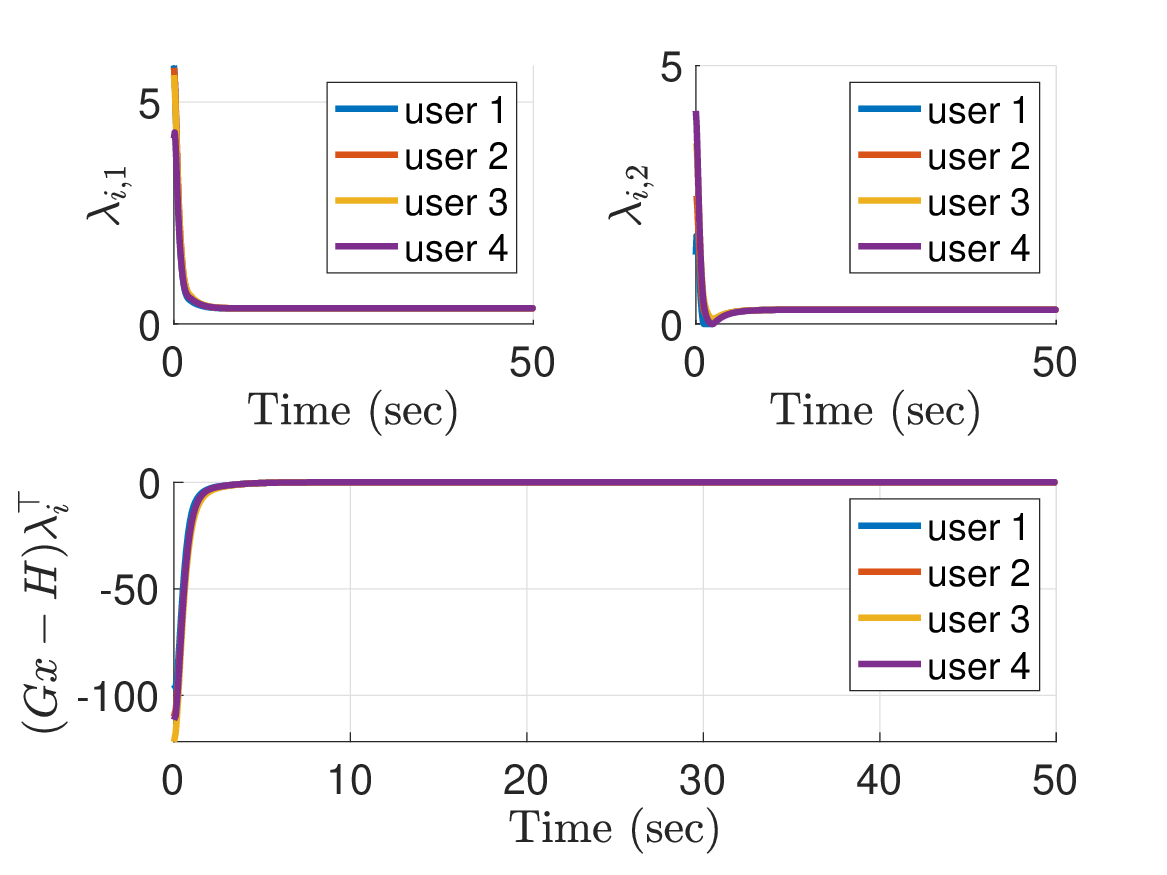}
        \caption{}
        \label{fig:full-b}
    \end{subfigure}
    
    \caption{Results under full-information setting: (a) The trajectories of $x_i$ , $v_i$ and $\max\{g_i(x)\}$. (b) The trajectories of the dual variables $\lambda_i$ and $(Gx-H)\lambda_i^\top$.}
    \label{fig:sim2}
\end{figure}
\begin{figure}[htbp]
    \centering
    \begin{subfigure}[b]{0.45\textwidth}
        \centering
        \includegraphics[width = 5.8cm]{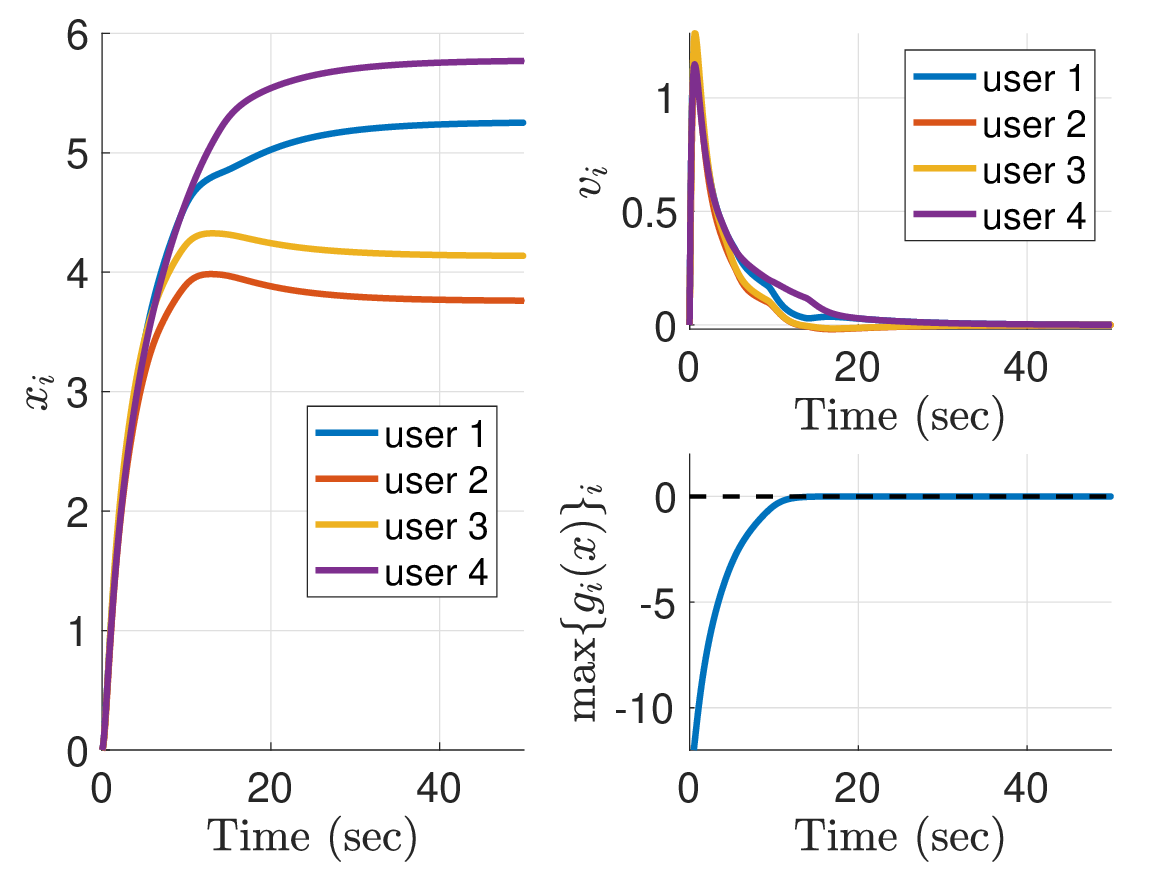}
        \caption{}
        \label{fig:partial-a}
    \end{subfigure}
    \begin{subfigure}[b]{0.45\textwidth}
        \centering
        \includegraphics[width = 5.8cm]{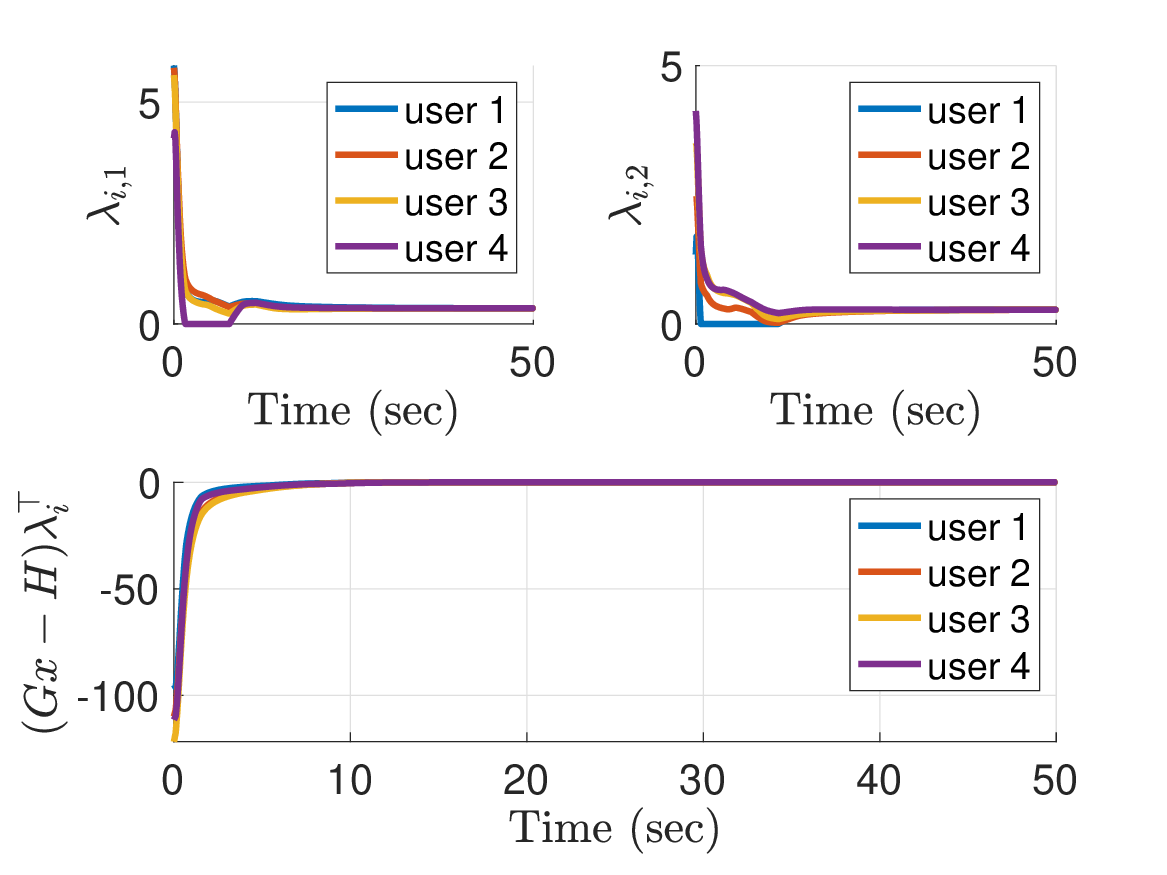}
        \caption{}
        \label{fig:partial-b}
    \end{subfigure}
    
    \caption{Results under partial-information setting: (a) The trajectories of $x_i$ , $v_i$ and $\max\{g_i(x)\}$. (b) The trajectories of the dual variables $\lambda_i$ and $(Gx-H)\lambda_i^\top$.}
    \label{fig:sim2-partial}
\end{figure}
\begin{figure}
    \centering
    \includegraphics[width = 7.5cm]{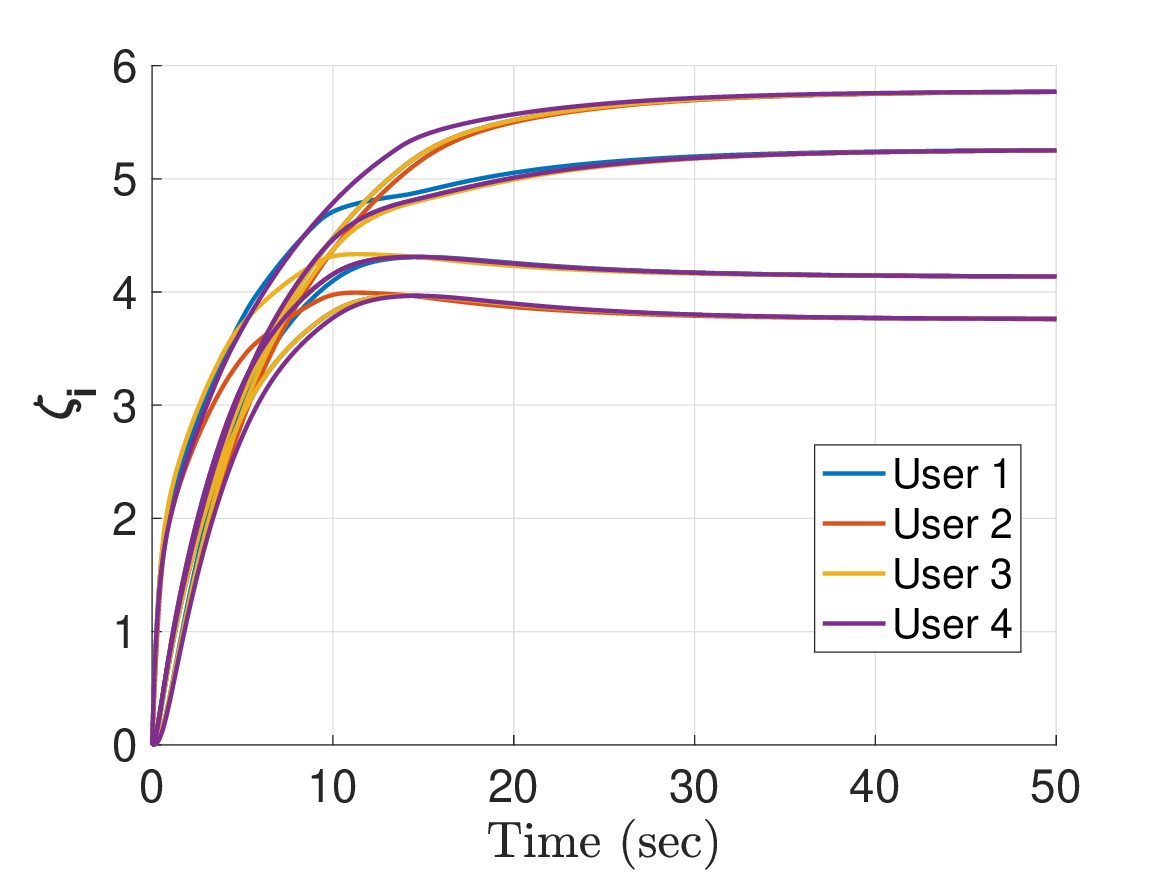}
    \caption{Augmented variable $\boldsymbol{\zeta}_i$ of every user.}
    \label{fig:estimation}
\end{figure}
\section{Conclusion}\label{sec:conclusion}
In this paper, we considered GNE seeking in non-cooperative games with coupled constraint sets. To enforce safety for distributed GNE seeking, we introduced a CBF-based design for the problem. We showed that under both full- and partial-information setting, the trajectories of the seeking dynamics stay within the constraint set for all time, the equilibrium of the dynamics is the exact v-GNE of the game and the dynamics asymptotically converge to the v-GNE. Furthermore, we extended the approach to games where the players are multi-integrator agents by utilizing a CBF-based coordinate transformation.

There are a number of open problems following from this work. While this paper considers affine coupled constraints, whether the CBF-based algorithm is also applicable for more general nonlinear constraints needs to be further investigated. The current work is based on a continuous-time setting, and how to introduce the methodology in a discrete-time framework remains an open question. How to generalize the algorithm design to solve for GNE seeking problems with general linear or nonlinear players is also a direction for future research. 

\bibliographystyle{plain}
\bibliography{references}

\end{document}